\documentclass[a4paper,12pt]{amsart}
\usepackage{amssymb,amsmath,amsfonts,a4wide, color, mathtools}
\usepackage{url}
\usepackage{relsize}

\usepackage{hyperref}
\hypersetup{
    colorlinks=true, 
    linktoc=all,     
   linkcolor=blue,  
}
\newtheorem{pro}{Proposition}[section]
\newtheorem{teo}[pro]{Theorem}
\newtheorem{lema}[pro]{Lemma}
\newtheorem{coro}[pro]{Corollary}
\theoremstyle{definition}
\newtheorem{defi}[pro]{Definition}
\newtheorem{rema}[pro]{Remark}
\newtheorem{ex}[pro]{Example}

\numberwithin{equation}{section}

\def \V{\mathcal{V}}
\def \H{\mathcal{H}}

\def \F{\mathcal{F}}

\def \r{\mathfrak{R}}
\def \D{\mathcal{D}}

\newcommand{\w}{\wedge}
\def \del{\partial}

\DeclareMathOperator{\bdel}{\overline{\partial}}
\DeclareMathOperator{\s}{\textsl{s}}
\DeclareMathOperator{\li}{\mathcal{L}}

\newcommand{\grad}{\textsl{grad\,}}

\newcommand{\bb}{\mathbb}
\DeclareMathOperator{\Ric}{\textsl{Ric}\,}
\def \z{\zeta}

\def\Ric{\mathrm{Ric}}
\def\s{\mathrm{scal}}
\def\HK{\mbox{HK}}
\def\grad{\mathrm{grad}}

\renewcommand{\ln}{\textsl{ln}\,}

\newcommand{\lie}[1]{\mathfrak{#1}}

\DeclareMathOperator{\aut}{\lie{aut}}

\DeclareMathOperator{\di}{d}
\DeclareMathOperator{\Ker}{Ker}

\newcommand{\ol}{\overline}
\newcommand{\ba}{\begin{array}}\newcommand{\ea}{\end{array}}

\renewcommand{\&}{{\footnotesize \&}}

\def\tg{\widetilde{\gamma}}
\def\hg{\widehat{\gamma}}
 
\def\sideremark#1{\ifvmode\leavevmode\fi\vadjust{\vbox to0pt{\vss
 \hbox to 0pt{\hskip\hsize\hskip1em
 \vbox{\hsize2.5cm\tiny\raggedright\pretolerance10000
 \noindent #1\hfill}\hss}\vbox to8pt{\vfil}\vss}}}%

\begin{document} 

\bibliographystyle{plain}

\title[Conformal Killing forms in K\"ahler geometry]{Conformal Killing forms in K\"ahler geometry}
\author[P.-A. Nagy]{Paul-Andi Nagy} 
\address[Paul-Andi Nagy]{Institut f\"ur Geometrie und Topologie, Fachbereich Mathematik, Universit\"at Stuttgart, Pfaffenwaldring 57, 70569 Stuttgart, Germany}
\email{Paul-Andi.Nagy@mathematik.uni-stuttgart.de}
\author[U. Semmelmann]{Uwe Semmelmann}
\address[Uwe Semmelmann]{Institut f\"ur Geometrie und Topologie, Fachbereich Mathematik, Universit\"at Stuttgart, Pfaffenwaldring 57, 70569 Stuttgart, Germany}
\email{Uwe.Semmelmann@mathematik.uni-stuttgart.de}
\date{\today} 
\begin{abstract}
For  K\"ahler manifolds we explicitly determine the solution to the conformal Killing form equation in middle degree. In particular,
we complete the classification of conformal Killing forms on compact K\"ahler manifolds. We give the first examples of conformal Killing forms on K\"ahler manifolds not coming from Hamiltonian $2$-forms. These are supported by Calabi type manifolds over a K\"ahler Einstein base. In this set up we 
also give structure results and examples for the closely related class of Hermitian Killing forms.
\end{abstract}
\maketitle

\section{Introduction} \label{intro-s}
Let $(M^{2m},g,J)$ be a connected K\"ahler manifold of complex dimension $m \geq 3$. The aim of this paper is to study the overdetermined system of first order equations on a pair 
$(\varphi,\tau)$ of differential forms given by 
\begin{equation}\label{special-i}
\nabla_X \varphi=X^{1,0} \wedge \tau \;+\; \frac{i}{2} \, \omega \wedge (X \lrcorner \, \tau)
\end{equation}
for all $X$ in $TM$ where $\varphi \in \Lambda^{1,m-1}_0M$ is a primitive form of complex type $(1,m-1)$. For such pairs, the form 
$\tau$ in  $\Lambda^{0,m-1}M$ is determined from $\tau=-\frac{2}{m+1}\,\del^*\varphi$. For explanation on notation and the various conventions see the body of the paper.
We will study equation \eqref{special-i} in full generality, in the sense that the metric $g$ is not assumed to be complete nor $M$ to be compact. 
\subsection{Motivation}
Equation \eqref{special-i} naturally appears when studying conformal Killing forms on K\"ahler manifolds, see \cite{AUT}, where its solutions are referred to as special $m$-forms. Indeed, solutions of
 \eqref{special-i} are precisely primitive conformal Killing forms in $\Lambda^{1,m-1}M$. 
 
Conformal Killing $p$-forms (or twistor forms) are differential $p$-forms in the kernel of a naturally defined conformally invariant 1st order elliptic operator, the so-called Penrose or twistor operator  \cite{Branson, Yano}. 
They generalise conformal vector fields for $p=1$.  Originally  conformal Killing forms,  are motivated from physics. The interesting subclass of Killing forms, i.e. co-closed conformal Killing forms,  were used in relativity theory,  for integrating the equations of motion \cite{Penrose1}. It turns out that Killing forms are related to several interesting geometric structures, e.g.  Sasaki, nearly K\"ahler or nearly parallel $G_2$ metrics \cite{uwe}. At the same time it could be shown  \cite{BMS,AUT,AU2, uwe2}  that Killing $p$-forms, $p\geq 2$, on compact manifolds with special holonomy, in particular compact K\"ahler manifolds, must be parallel. Later on conformal Killing forms became important in the study of symmetries of the massless Dirac equation \cite{Benn} and also in mathematics  \cite{AD,A19}.  Many interesting non-parallel examples can be found, e.g. on compact K\"ahler manifolds as the complex projective spaces. 

In  \cite{AUT} conformal Killing forms on a compact K\"ahler manifold of real dimension $2m$ with $m\geq 3$, have been classified up to degrees $2$ and $m$. 
In degree two conformal Killing forms turn out to be in $1:1$ correspondence with Hamiltonian $2$-forms. The latter have been intensively studied in \cite{ACG} and the follow-up papers. In particular one has local and global classification results, as well as many interesting examples which underpin the geometry of specific Hamiltonian torus actions.
In degree $m$ the classification of conformal Killing forms remains an open problem. 
By results in \cite{AUT} it amounts to (see section \ref{Cm}) determining, when 
forms $\varphi$ in $\Lambda^{1,m-1}_0M$ solving the equation \eqref{special-i} exist. 
In our article we will solve this problem and complete the classification of conformal Killing forms on 
compact K\"ahler manifolds. Moreover, we will produce first examples not coming from Hamiltonian $2$-forms.

The conformal Killing  equation and in particular   \eqref{special-i} is an overdetermined system of PDEs of finite type; its prolongation has been considered in \cite{uwe}. This is indication that an explicit procedure
for constructing solutions may exist. We build  it in several stages  relying on the key observation that
 $\tau$ is a Hermitian Killing form in the sense that 
\begin{equation} \label{HKD}
\nabla^{01}\tau=\frac{1}{m}\bdel \tau,
\end{equation}
which aditionally satisfies $\di\!^{\star} \tau=0$.

On the other hand \eqref{HKD} is a natural generalisation of the Killing equation in K\"ahler geometry 
for the operator $\nabla^{01}-\frac{1}{m} \bdel $  is precisely the projection of $\nabla-\frac{1}{m}\di$ onto 
$T^{01}M \otimes \Lambda^{0,m-1}M$. Hermitian Killing forms have already been considered in \cite{handbk} in order to describe the structure of the torsion of $\mathcal{G}_1$-manifolds and are of independent interest.

First order properties of the space $\HK^{0,m-1}(M,g)$ of Hermitian Killing forms of type $(0,m-1)$, including examples, are described in section \ref{defn}, where the main observation is that the non-vanishing of $\HK^{0,m-1}(M,g) \cap \ker \di^{\star}$ forces the Ricci tensor of $g$ to have at most two eigenfunctions over $M$. In addition on 
the open part of $M$ where $g$ is non-Einstein one eigenfunction has multiplicity $2$. This eigenvalue type for the Ricci tensor has been 
considered, for constant eigenfunctions, in \cite{ADM}; it also appears as an integrability condition for several classes of PDE`s of geometric origin such as K\"ahler metrics conformal to Einstein \cite{Der-M} and Hamiltonian forms $2$-forms of rank $1$ \cite{ACG}.

\subsection{Main results}
Our first main result below consists in giving the complete 
local structure of solutions $(\varphi,\tau)$ to \eqref{special-i} together with that of the K\"ahler structure $(g,J)$.
\begin{teo} \label{main1}
Let $(M^{2m},g,J), m \geq 3$ be a connected K\"ahler manifold admitting a solution $(\varphi,\tau)$ to \eqref{special-i} such that 
$\bdel \tau$ is not identically zero. Then either $\varphi$ is parallel 
w.r.t. the Levi-Civita connection of $g$ or one of the following situations occur
\begin{itemize}
\item[(i)] the metric $g$ is Ricci flat and equipped with a cone vector field $V_M$. Up to adding parallel forms to $\varphi$ we have 
\begin{equation*}
\varphi=\frac{1}{4}V_M^{1,0} \w (V_M \lrcorner \, \Psi_M), \tau=V_M \lrcorner \, \Psi_M
\end{equation*}
where $\Psi_M$ is a normalised complex volume form
\medskip
\item[(ii)] $(g,J)$ is locally of Calabi type, with local K\"ahler Einstein base $N$, moment map $z$ and momentum profile $\bb{X}(z)=z(C_1z^m+\frac{2k}{m})$
with $C_1,k \in \bb{R}$. Then 
\begin{equation*}
\varphi=\frac{z}{\bb{X}(z)}\del z \w \tau, \ \tau=z^m \widehat{id}. 
\end{equation*}
\end{itemize} 
\end{teo}
The K\"ahler structures occuring in (i) above are locally metric cones over Sasaki-Einstein manifolds respectively 
the conification \cite{MaRo} of a $2(m-1)$-dimensional K\"ahler-Einstein manifold. See section \ref{Einstein} for details.
The definition and main properties of Calabi-type metrics as in part (ii) of Theorem \ref{main1} are explained in 
sections \ref{CAL} and \ref{defn-cal} of the paper. The form $\widehat{id}$ is a canonically defined $(0,m-1)$-form on $M$ build with the aid of the Einstein condition on $N$ and its polarisation $L$(see section \ref{lifts}).
Cases (i) and (ii) above do not overlap since in the latter the metric is never Einstein. The geometry of K\"ahler structures $(g,J)$ carrying a solution $(\varphi,\tau)$ to \eqref{special-i} such that $\bdel \tau=0$ has a different flavour and cannot be treated with the current techniques. However we display a non-trivial class of such solutions in the Ricci flat case (see section \ref{Einstein}). Moreover when $M$ is compact it is a simple observation that  $\bdel \tau=0$ forces $\varphi$ to be parallel.

Based on Theorem \ref{main1} we complete the classification of conformal Killing forms on compact K\"ahler manifolds. 
\begin{teo} \label{main2}
Let $(M^{2m},g,J), m \geq 3$ be a compact K\"ahler manifold. 
Any primitive conformal Killing form in $\Lambda^{1,m-1}M$ is parallel w.r.t. 
to the Levi-Civita connection of $g$. Moreover, up to parallel forms, the space of conformal Killing forms with degree $\neq 1,2m-1$ is  
$$ \{L_{\omega}^{k-1}(\psi-\frac{1}{2k}\langle \psi, \omega\rangle \omega) : \psi \ \mbox{is Hamiltonian}, \ 1 \leq k \leq m-1 \}.
$$
where $L_{\omega}$ denotes exterior multiplication with the K\"ahler form $\omega$. 
\end{teo}

In dimension $4$ it is an open problem, even for compact K\"ahler manifolds, to fully classify conformal Killing $2$-forms. Partial results, including many examples not coming from Hamiltonian $2$-forms as well as the relation with ambi-K\"ahler geometry can be found in \cite{Mas,AUT,AP}.
\subsection{Outline of proofs}

The key ingredient in proving Theorem \ref{main1} is to derive the full set of integrability conditions for \eqref{special-i}. To that extent we use foliation theory as follows. The distribution 
$\V:=\ker(\tau)$ has complex rank $1$ off the zero set of $\tau$. Denoting $\H:=\V^{\perp}$, the integrability conditions mentioned above amount to showing that the components 
$(\del \tau)_1$ respectively $(\del \tau)_3$ of $\del \tau$ on the subbundles $(\omega^{\V}-\omega^{\H}) \w \Lambda^{0,m-2}\H \subseteq \Lambda^{1,m-1}_0M$ respectively $\Lambda^{0,1}\V \w \Lambda^{1,m-2}_0\H \subseteq \Lambda^{1,m-1}_0M$ vanish off the set where $g$ is Einstein. Geometrically these mean that $\V$ defines a totally geodesic, holomorphic and conformal foliation w.r.t. $(g,J)$. By results in \cite{Chiossi-Nagy}(see also \cite{NaOr} for instances when $\V$ has arbitrary rank) this entails that $\V$ induces a canonically defined symmetry $K$ and at the same time guarantees that $(g,J)$ is locally of Calabi type. During this process we also show, just as in the case of $\del \tau$, that the components $\varphi_1$ and $\varphi_3$ of $\varphi$ vanish, which makes that equation \eqref{special-i} essentially reduces to a first order ODE for the moment profile 
of $K$. Solving this ODE 
reveals that $K$ has momentum profile as indicated in Theorem \ref{main1}, (ii) and also provides the explicit expression for $\varphi$ therein.
On open sets where $g$ is Einstein $-J\grad \vert \tau \vert^2$ is a cone vector field and Theorem \ref{main1} is proved by an easy direct argument.

To prove Theorem \ref{main2} we show that, {\it{a posteriori}},
\begin{equation} \label{symh}
p:=\frac{\s_g}{m(2m-1)} \vert \varphi \vert^2+\vert \tau \vert^2
\end{equation}
is a real holomorphy potential on $M$ in the sense that $-J\grad p$ is a holomorphic Killing vector field. 
To see this we use the local symmetry $K$ constructed above together with the unique continuation property for $\tau$ which allows fixing the constants of integration over $M$ and thus ensures that $p$ is globally defined. In addition, we observe that 
$p$ satisfies a second order non-linear PDE, see \eqref{max}; this leads to the claimed rigidity result by using the maximum principle for the Laplace operator on functions.

\subsection{Further results}
In section \ref{ex} we examine global aspects pertaining to the theory of Hermitian Killing forms. Starting from a polarised K\"ahler manifold $(N^{2(m-1)},L)$ we consider the K\"ahler manifold $M^{2m}=L$ equipped with a Calabi type metric and calculate explicitely the space $\HK^{0,m-1}(M,g)$. The following theorem shows that there are large classes of K\"ahler metrics, including complete ones, such that the space of Hermitian Killing has a rich structure.
\begin{teo} \label{main3}
We have an isomorphism 
\begin{equation*}
\bigoplus_{k \in \bb{Z}}H^0(N,K_N \otimes L^k) \to \{\tau \in \HK^{0,m-1}(M,g): \tau_{\vert \V}=0\}.
\end{equation*}
\end{teo}
Above $K_N$ denotes the canonical line bundle of $N$ and $H^0(N,E)$ indicates the space of holomorphic sections of a holomorphic bundle $E$. The vertical 
distibution distribution of the fibration $\bb{C} \hookrightarrow M \to N$ is denoted by $\V$.
\begin{rema}
We believe that the foliation inspired techniques developed in this paper could be applied to solve other PDEs of geometric origin. In particular to treat the imaginary K\"ahlerian Killing spinor equation, which is essentially a coupled equation for Hermitian Killing forms with values in a complex line bundle. See \cite{Kirchberg,GS} for details and partial classification results. The lifting technique developed in section \ref{ex} is expected to play a role in this direction.
\end{rema} 
\section{Preliminaries} \label{prel-s}
\subsection{Conformal Killing forms} \label{CFK-s}
Let $(M^n,g)$ be a Riemannian manifold which is furthermore assumed to be oriented by a volume form $\nu$ in $\Lambda^nTM$. A 
differential $p$-form $\varphi$ in $\Lambda^pTM$ is called a {\it{conformal Killing form}} if and only if 
\begin{equation} \label{tweq}
\nabla_X \varphi=\frac{1}{p+1}X \lrcorner \, \di\! \varphi-\frac{1}{n-p+1} X^{\flat} \wedge \di^{\star}\! \varphi
\end{equation}
for all $X$ in $TM$. Another interpretation of the equation involved in the definition of conformal Killing forms is 
through the decomposition of the tensor product $\Lambda^1 \otimes \Lambda^p$. Let $a : \Lambda^1 \otimes \Lambda^p \to \Lambda^{p+1}$ be the total antisymmetrisation map 
and let the trace map $t : \Lambda^1 \otimes \Lambda^p \to \Lambda^{p-1}$ be given by 
$
t(\varphi)=\sum \limits_{i}e_i \lrcorner \, \varphi_{e_i}
$
for all $\varphi$ in $\Lambda^1 \otimes \Lambda^p$, where $\{e_i\}$ is a local orthonormal frame of $TM$. 
We have then an orthogonal decomposition 
$\Lambda^1 \otimes \Lambda^p=\Lambda^{p+1} \oplus \Lambda^{p-1} \oplus (\ker(a) \cap \ker(t))$
and hence \eqref{tweq} is  equivalent with the vanishing of the component of $\nabla \varphi$ on $\ker(a) \cap \ker(t).$ We also define, for subsequent use, the algebraic commutator 
$ 
[\alpha, \varphi]=\sum \limits_{i} (e_i \lrcorner \, \alpha) \wedge (e_i \lrcorner \, \varphi)
$
for all $\alpha$ in $\Lambda^2TM$ and for all $\varphi$ in $\Lambda^{\star}TM$.

The rest of this section is devoted to recall various characterisations and properties of conformal Killing forms. Let $R$ be the Riemann curvature tensor of $g$, defined by $R(X,Y)Z=\nabla^2_{Y,X}Z-\nabla^2_{X,Y}Z$.  The curvature operator 
$\r: \Lambda^{\star}TM \to \Lambda^{\star}TM$ of the metric $g$ is given by 
$$ 
\r \varphi=\sum \limits_{i,j} R(e_i, e_j) \wedge( e_i \lrcorner \, e_j \lrcorner \, \varphi)
$$
for all $\varphi$ in $\Lambda^{\star}TM$. It preserves the degree of forms and vanishes on 
$\Lambda^1TM \oplus \Lambda^{n-1}TM$. The Laplacian $\Delta=
\di\di^{\star}+\di^{\star}\di$ acting on differential forms is related to the rough Laplacian by 
\begin{equation} \label{wz}
\Delta=\nabla^{\star} \nabla +\r +\Ric. 
\end{equation}
Here $\Ric=\sum \limits_{}R(e_i, \cdot, e_i, \cdot)$ is the Ricci tensor of the metric $g$, acting on forms $\varphi$ in $\Lambda^{\star}TM$ 
according to $ 
\Ric(\varphi)=\sum \limits_{i} \Ric (e_i) \wedge (e_i \lrcorner \, \varphi).$
Differentiating the conformal Killing form equation shows, after using the Weizenb\"ock formula \eqref{wz}, that any conformal Killing form $\varphi$ in $\Lambda^pTM$ satisfies 
\begin{equation} \label{lap}
\frac{p}{p+1}\di^{\star}\!\di\!\varphi+\frac{n-p}{n-p+1}\di\!\di^{\star}\! \varphi=\r\varphi+\Ric (\varphi). 
\end{equation}
If moreover $M$ is compact, the converse has been proved in \cite{uwe}.
In this paper we will mainly use the following two facts.
\begin{pro} \label{der12} 
Let $\varphi$ be a conformal Killing $p$-form and $X$ any vector field. Then 
\begin{itemize}
\item[(i)]
$\frac{n-p}{n-p+1}\nabla_X(\di^{\star}\!\varphi)+\frac{1}{p+1}X \lrcorner \, \!\di^{\star}\di\!\varphi=\Ric X \lrcorner \, \varphi+\frac{1}{2} (X \lrcorner \, \r \varphi-\r(X \lrcorner \, \varphi))$ \\
\item[(ii)] $\di(\r \varphi)=\frac{p-1}{p+1} \r(\di\!\varphi)$. 
\end{itemize}
\end{pro}
\begin{proof}
(i) has been proved in \cite{uwe}. For proving (ii) we compute:
$$
\di(\r \varphi) \;=\;  \sum \limits_{i}e_i \wedge \nabla_{e_i} (\r \varphi)
=  \sum \limits_{i,j,k} e_i \wedge \biggl [ (\nabla_{e_i}R)(e_j, e_k) \wedge e_j \lrcorner \, e_k \lrcorner \, \varphi+R(e_j,e_k) \wedge e_j \lrcorner \, e_k \lrcorner \, \nabla_{e_i}\varphi\biggr ].
$$
The first summand above vanishes by the differential Bianchi identity written under the form 
$ \sum  e_i \wedge (\nabla_{e_i}R)(X,Y)=0$
for all $X,Y$ in $TM$. Using $A=\frac{1}{p+1}\di\!\varphi, B=-\frac{1}{n-p+1}\di^{\star}\! \varphi$ as a shorthand notation and taking into account the conformal Killing form equation for $\varphi$
we are left with 
$$
\begin{array}{ll}
\di(\r \varphi)
& =
 \sum \limits_{i,j,k}e_i \wedge R(e_j,e_k) \wedge e_j \lrcorner \, e_k \lrcorner \, e_i \lrcorner \, A 
\;+\;  \sum \limits_{i,j,k}e_i \wedge R(e_j,e_k) \wedge e_j \lrcorner \, e_k \lrcorner \, (e_i \wedge B)\\[2ex]
&=
(p-1) \r A \;+\; \sum \limits_{i,j,k}e_i \wedge R(e_j,e_k) \wedge e_j \lrcorner \, e_k \lrcorner \, (e_i \wedge B). 
\end{array}
$$
Since 
$\,
 e_j \lrcorner \, e_k \lrcorner \, (e_i \wedge B)=\langle e_k, e_i\rangle e_j \lrcorner \, B- \langle e_i, e_j \rangle e_k \lrcorner \, B+e_i \wedge (e_j \lrcorner \, e_k \lrcorner \, B)$
the last summand above equals 
$$
\sum \limits_{i,j} e_i \wedge R(e_j, e_i) \wedge e_j \lrcorner \, B-\sum \limits_{i,k} e_i \wedge R(e_i, e_k) \wedge e_k \lrcorner \, B 
$$
and hence it vanishes after using the Bianchi identity under the form $ \sum  e_i \wedge R(e_i ,X)=0$ 
for all $X$ in $TM$. The claim is now proved.
\end{proof}
In particular a non-trivial obstruction to the existence of a conformal Killing $p$-form $\varphi $, $p \geq 2$ is 
\begin{equation} \label{obs}
\di\! \r \di\!\varphi=0,
\end{equation}
obtained by differentiating in (ii) above. It will play a key role later on in the paper.
\subsection{Elements of K\"ahler geometry} \label{K-prel}
Let $(M^{2m},g,J)$ be a K\"ahler manifold with K\"ahler form $\omega=g(J \cdot, \cdot)$. The complex structure  $J$ acts on forms via
\begin{equation*}
J \alpha(X_1, \ldots, X_p)=\alpha(JX_1, \ldots, JX_p).
\end{equation*}
Letting $\Lambda^{p,q}TM$ be the space of (complex valued) forms of bidegree $(p,q)$ recall that $[\omega, \varphi]=i(q-p) \varphi $
for all $\varphi$ in $\Lambda^{p,q} M$. For any $X$ in $TM$ we write 
$$
X_{1,0}=\frac{1}{2}(X-iJX) , \ X_{0,1}=\frac{1}{2}(X+iJX)
$$
for the components in $T^{1,0}M$ respectively $T^{0,1}M$. Letting
$
()^{\flat}:TM \to \Lambda^1TM
$ 
be the isomorphism induced by the metric, we have for any $X$ in $TM$ one forms 
$$ 
X^{1,0}:=\frac{1}{2}(X^{\flat}+i(JX)^{\flat}), \ X^{0,1}:=\frac{1}{2}(X^{\flat}-i(JX)^{\flat})
$$
in $\Lambda^{1,0}M$ respectively $\Lambda^{0,1}M$. For any vector field $X$ and any $\alpha$ in 
$\Lambda^{p,q}M$  we write 
\begin{equation} \label{proj}
(X \lrcorner \, \alpha)_{\Lambda^{p,q-1}M}=X_{0,1} \lrcorner \, \alpha, \ (X \lrcorner \, \alpha)_{\Lambda^{p-1,q}M}=X_{1,0} \lrcorner \, \alpha. 
\end{equation}
for the orthogonal projections.
Denote by $L_{\rho}$ respectivelly $L_{\omega}$ the exterior multiplication with the Ricci form $\rho=g(\Ric J \cdot, \cdot)$ respectively the K\"ahler form $\omega$. An algebraic fact we will use in what follows is
\begin{equation} \label{ls}
[L^{\star}_{\omega},L_{\rho}]=\frac{\s}{2}-\Ric 
\end{equation}
on $\Lambda^{\star}TM$, where $\s$ denotes the scalar curvature of the metric $g$. Because $J$ is a complex structure the exterior derivative splits as 
$\di=\partial+\overline{\partial}$
where $\partial : \Lambda^{p,q}M \to \Lambda^{p+1,q}M$ and $\overline{\partial} : \Lambda^{p,q}M \to \Lambda^{p,q+1}M$. From $\di^2=0$ it follows that 
$ \partial^2=\overline{\partial}^2=0, \ \partial \overline{\partial}+\overline{\partial} \partial=0.$
The covariant derivative $\nabla$ of the metric $g$ splits as $\nabla=\nabla^{1,0}+\nabla^{0,1}$ and moreover 
$\nabla^{1,0}=\partial $ on $\Lambda^{0,\star}M$. Since $(g,J)$ is K\"ahler we also have the so-called {\it K\"ahler identities}
\begin{equation} \label{ki}
[L^{\star}_{\omega}, \overline{\partial}]=-i{\partial}^{\star},  \quad  \ [L^{\star}_{\omega}, \partial]=i{\overline{\partial}}^{\star}
\end{equation}
together with their dual version 
\begin{equation} \label{kid}
[\overline{\partial}^{\star},L_{\omega}]=i\partial, \quad  \ [\partial^{\star},L_{\omega}]=-i\overline{\partial}
\end{equation}

Finally we recall for later use the following Weitzenb\"ock formula on $(p,q)$-forms 
\begin{equation}\label{wbf}
\Delta_{\bar\del}=(\nabla^{0,1})^* \nabla^{0,1} \;-\;i \,[\rho, \cdot] \ .
\end{equation}
The curvature endomorphism $\,- i [\rho, \cdot] \,$ acts as $\,\Ric \,$ on forms of type $(0,q)$ and by
multiplication with the function $\frac12 \,\s$ in the special case of forms of type $(0,m)$, where $m = \dim_{\mathbb C} M$.

We end this section  by recalling some curvature related facts. The operator $\r$ preserves complex type and vanishes on $\Lambda^{0, \star}M \oplus \Lambda^{\star,0}M$. Moreover, by straightforward computation  
\begin{equation} \label{cc1}
[\r, L_{\omega}] = -2L_{\rho} 
\end{equation}
on $\Lambda^{\star}TM$.
 
Throughout this paper we indicate with $\aut(M,J):=\{X:\li_XJ=0\}$ the Lie algebra of real holomorphic vector fields and with 
$\aut(M,g,J):=\{X \in \aut(M,J):\li_Xg=0 \}$ the Lie algebra of holomorphic Killing vector fields.
\section{Hermitian twistor and Killing forms} \label{defn}
\subsection{Hermitian twistor forms on K\"ahler manifolds} \label{sect-prol}
\begin{defi} \label{def1}
A  {\it Hermitian twistor form} on a K\"ahler manifold $M$ is a form  $\varphi$ in $\Lambda^{p,q}M$
satisfying for all vector fields $X$ the equation
$$
 \nabla_X \varphi=(X \wedge B+X \lrcorner \, A)_{\Lambda^{p,q}TM}
$$
for some forms $A,B$ in $\Lambda^{\star}TM$. If $B$ vanishes $\varphi$ is called a {\it Hermitian Killing form}.
\end{defi}
The subscript above indicates orthogonal projection onto the space of $(p,q)$-forms. The space of Hermitian Killing forms 
of type $(p,q)$ is denoted with $\HK^{p,q}(M,g)$. Hermitian Killing forms have been introduced in \cite{handbk} in the general  framework of almost-Hermitian manifolds. It turns out that conformal Killing forms on K\"ahler manifolds, and in particular special forms (see next section), are a subclass of  the more flexible class of Hermitian twistor forms. The following equivalent definition will be useful later on.

\begin{lema}
A form $\varphi$ in $\Lambda^{p,q}M$ is a Hermitian Killing form if and only if
\begin{equation}\label{HK}
\nabla_X \varphi \; =\;  \frac{1}{p+1} \, X_{1 0 }\, \lrcorner \, \, \del \phi \;+\; \frac{1}{q+1}X_{01}\, \lrcorner \, \, \bar\del \phi
\end{equation}
for all vector fields $X$.
\end{lema}

To describe the local normal forms and further properties of elements in $\HK^{0,q}(M,g)$ some preliminary observations are needed. 
Letting 
$\mathfrak{h}^{0,q}(M,g):=\{\alpha \in \Lambda^{0,q}M : \nabla^{0,1}\alpha=0 \}$
we have 
\begin{equation*} 
\HK^{0,q}(M,g) \cap \ker \overline{\partial}=\mathfrak{h}^{0,q}(M,g)
\end{equation*}
for $0 \le q \le m$, with the special case $\; \HK^{0,m}(M,g)=\mathfrak{h}^{0,m}(M,g)$. The Lie algebra of real holomorphic vector fields 
$\mathfrak{aut}(M,J)$ is dual to the space $\mathfrak{h}^{0,1}(M,g)$ in the sense that 
$X \in \mathfrak{aut}(M,J) \mapsto X^{0,1} \in \mathfrak{h}^{0,1}(M,g)$
is a linear isomorphism. Let  $\alpha$ be in $\mathfrak{h}^{0,1}(M,g)$ then 
a local {\it{holomorphy potential}} for $\alpha$ is a function $f$  such that $\bdel f=\alpha$. It is unique  
up to addition of holomorphic functions, and its existence is granted by the Dolbeault Lemma.
\begin{pro} \label{norm-form-1}
Let $\tau$ belong to $\HK^{0,m-1}(M,g)$. Around any point where $\bdel \tau \neq 0$ let $\alpha_i, 1 \le i \le m$ be a local basis 
in $\mathfrak{h}^{0,1}(M,g)$ with local holomorphy potentials $f_i, 1 \le i\le m$. Then:
$$ 
\tau=G \sum \limits_{i=1}^{m} (-1)^{i-1}f_i \bdel f_1\,  \wedge \ldots \wedge \widehat{\bdel f_i} \wedge \ldots \bdel f_m
$$
for some (localy defined) holomorphic function $G$.
\end{pro}
\begin{proof}
Write $\tau =\sum \limits_{i=1}^m F_i \alpha_1 \wedge \ldots \wedge \widehat{\alpha}_i \wedge \ldots \wedge \alpha_m$ and $\bdel \tau=H \alpha_1 \wedge \ldots \wedge \alpha_m$ with 
$H \neq 0$. Expanding $\nabla^{0,1}\tau=\frac{1}{m}\bdel \tau$ gives $\bdel F_i=\frac{(-1)^{i-1}H}{m} \alpha_i$. Moreover, applying $\bdel$ shows that $\bdel H \wedge \alpha_i
=0, 1 \le i \le m$ hence $\bdel H=0$. The functions $f_i=\frac{m(-1)^{i-1}F_i}{H}$ are holomorphy potentials for $\alpha_i$ and the claim follows with $G=\frac{H}{m}$. 
\end{proof}

As already mentioned the concept of Hermitian twistor forms is much more flexible compared to conformal
Killing forms. In particular one has many examples of compact K\"ahler manifolds admitting Hermitian Killing forms, which can be obtained as follows
\begin{ex} \label{eg2}
Let $\{K_i ,1 \le i \le p \le m \}$ be Hamiltonian Killing vector fields for $(M^{2m},g,J)$ with moment maps $t_i$, that is $\li_{K_i}\!g=0$ and $K_i \lrcorner \, \omega=dt_i$. 
Then
$$ \tau=\sum \limits_{i=1}^{p} (-1)^{i-1}t_i \bdel t_1 \wedge \ldots \wedge \widehat{\bdel t_i} \wedge \ldots  \wedge \bdel t_p 
$$
is a Hermitian-Killing form of type $(0,p-1)$. In particular any compact toric K\"ahler manifold admits such forms.
\end{ex}

We will describe first the holomorphic prolongation of the Hermitian Killing equation~\eqref{HK}, in the
particular case of forms $\tau$ of type $(0,m-1)$.
We first note that \eqref{wbf} implies
\begin{equation} \label{hol-wz}
\mathfrak{h}^{0,q}(M,g) \subseteq \ker (\Delta_{\bdel}+i[ \rho, \cdot ])
\end{equation}
with equality when $M$ is compact; in particular 
\begin{equation} \label{hol-top}
\mathfrak{h}^{0,m}(M,g) \subseteq \ker (\Delta_{\bdel}-\tfrac{\s}{2}).
\end{equation}
\begin{pro} \label{prol-1}
Let $(\tau,X)$ belong to $\HK^{0,m-1}(M,g) \times \Gamma(TM)$. Then
\begin{itemize}
\item[(i)] $\overline{\partial} \tau \in \mathfrak{h}^{0,m}(M,g)  
\ \mbox{and} \ 
\Delta(\bar\del \tau) = \s \,\bar\del \tau$
\item[(ii)] $\nabla_{X_{0,1}}\partial \tau = 
\frac{1}{2} \r (X^{\flat} \wedge \tau)-\frac{1}{m} X_{0,1} \,\lrcorner \,\, \partial \overline{\partial} \tau $
\item[(iii)] 
$\nabla_{X_{0,1}}\overline{\partial}^{\star} \tau =  
\Ric X\, \lrcorner \,\, \tau -\frac{1}{m} X_{0,1} \,\lrcorner \,\, \overline{\partial}^{\star} \overline{\partial} \tau $
\item[(iv)] 
$\del^{\star} \del \tau=\Ric(\tau)+ \frac{1}{m} \bdel^{\star} \bdel \tau. $
\end{itemize}
\end{pro}
\begin{proof}
We prove (i) and (ii) at the same time. Let $\{e_k\}$ be a local orthonormal frame on $M$ with 
the dual basis $\{e^k\}$ of $1$-forms.
Then differentiating  in \eqref{HK}  for a Hermitian Killing form $\,\tau\,$ of type $\,(0,m-1)\,$
yields 
$$
\begin{array}{ll}
\sum e^k \wedge \nabla^2_{e_k,X} \tau
& =
\sum e^k \wedge (X_{1,0} \lrcorner \, \nabla_{e_k} \partial \tau) \;+\; \sum \frac{1}{m} e^k \wedge (X \lrcorner \, \nabla_{e_k} \overline{\partial} \tau) \\[1.5ex]
& = 
(\nabla_{X_{1,0}}\del \tau - X_{1,0} \lrcorner \, \bdel \del \tau ) \;+\;
\frac1m (\nabla_X(\bdel \tau)-X \lrcorner \, \del \bdel \tau )
\end{array}
$$
Using the Ricci identity and the curvature relation
\begin{equation} \label{id-c1}
\r (X^{\flat} \wedge \varphi) \;=\; 2 \sum  R(e_i,X) \wedge e_i \lrcorner \, \varphi  
\end{equation}
where $\varphi$ is a form in $ \Lambda^{0,p}M$ we get for the l.h.s 
\begin{equation*}
\sum e^k \wedge \nabla^2_{e_k,X} \tau=\nabla_X(\di\!\tau)-\sum e^k \wedge [R(e_k,X),\tau]=
\nabla_X(\di\! \tau)-\tfrac{1}{2}\r(X^{\flat} \wedge \tau). 
\end{equation*}
Substituting this into the equation above yields 
\begin{equation*}
\nabla_X(\di\!\tau)-\tfrac{1}{2} \r (X^{\flat} \wedge \tau)=
\nabla_{X_{1,0}} \partial \tau-X_{1,0} \lrcorner \, \bdel \del \tau +\tfrac{1}{m} ( 
\nabla_X(\bdel \tau)-X \lrcorner \, \del \bdel \tau). 
\end{equation*}
Since $\del$ and $\bdel$ anti-commute expanding $\di=\del+\bdel$ leads easily to 
\begin{equation*}
\nabla_{X_{0,1}} \del \tau+(1-\tfrac{1}{m}) \nabla_X(\bdel \tau)
\;=\;
\tfrac{1}{2} \mathfrak{R}(X^{\flat} \wedge \tau)-\tfrac{1}{m}X_{0,1} \lrcorner \, \del \bdel \tau+
(1-\tfrac{1}{m}) X_{1,0} \lrcorner \, \del \bdel \tau. 
\end{equation*}
Type considerations, together with the vanishing of $\mathfrak{R}$ on $\Lambda^{0, \star}M$ show that 
$\mathfrak{R}(X^{\flat} \wedge \tau)$ and $X_{0,1} \lrcorner \, \del \bdel \tau$ belong to $\Lambda^{1,m-1}M$ while $X_{1,0} \lrcorner \, \del \bdel \tau$ is in $\Lambda^{0,m}M$. 
The claims in (i) and (ii) follow by projection onto $\Lambda^{0,m}M$ respectively $\Lambda^{1,m-1}M$.\\
(iii) follows by applying $L^{\star}_{\omega}$ in (ii) and using the K\"ahler identities as well as 
\begin{equation*}
 L^{\star}_{\omega} \partial \overline{\partial} \tau = i \, \overline{\partial}^{\star} \overline{\partial} \tau 
 \qquad\mbox{and}\qquad
 L^{\star}_{\omega} \r(X \wedge \tau)=2i \Ric X \lrcorner \, \tau.
\end{equation*}
(iv) is proved by contracting (ii) while  taking into account the three equations 
$$ 
-\sum e_k \lrcorner \, \nabla^{0,1}_{e_k}\del \tau=\del^{\star} \del \tau, \;
 -\sum e_k \lrcorner \, (e_k)_{0,1} \lrcorner \, \del \bdel \tau \,=\, \bdel^{\star} \bdel \tau, \;
  \sum e_k \lrcorner \, \r(e_k \w \tau) \,=\,-2\Ric(\tau) \ .
 $$
\end{proof}
Recall that a smooth section of some vector bundle $E \to M$ is said to have the {\it strong unique continuation 
property} if it vanishes over $M$ as soon as it vanishes over some open non-empty subset of $M$. 
\begin{coro}\label{zero}
If $M$ is connected, any $ \tau$ in $\HK^{0,m-1}(M,g)$ has the strong unique continuation property. In particular, 
if $\tau $ is not identically zero the set $D = \{ x \in M : \tau_x \neq 0\}$ is dense in $M$.
\end{coro}
\begin{proof}
Assuming $\tau = 0 $ on an open set $U\subset M$ we get $\bar\del \tau = 0$ on $U$. 
But $\bar\del \tau$ belongs to $\mathfrak{h}^{0,m-1}(M,g)$ by Proposition~\ref{prol-1} and hence has the strong unique continuation property; to see this it is enough to record that with respect to any local frame $\{\alpha_i,1 \leq i \leq m\}$ in $\Lambda^{0,1}M$ which is the metric dual of a local frame of real holomorphic vector fields we have $\bdel \tau=H\alpha_1^{0,1} \wedge \ldots \w \alpha_{m}^{0,1}$ where $H$ is some locally defined holomorphic function. Thus $\bdel \tau$ vanishes over $M$, in other words $\tau \in \mathfrak{h}^{0,m-1}(M,g)$. As above (see also the proof of Proposition \ref {norm-form-1}), since $\tau$ has holomorphic coefficients in any local 
basis of $\Lambda^{0,m-1}M$ induced by a local basis in $\mathfrak{h}^{0,1}(M,g)$, it follows that $\tau$ is identically zero. Consequently if $\tau$ is not identically zero the set $\{x \in M: \tau_x=0\}$ has empty interior and the second part of the claim follows.
\end{proof}

As another consequence of  Proposition \ref{prol-1} we obtain an eigenvalue restriction for the Ricci tensor on manifolds with Hermitian Killing forms of
type $(0,m-1)$. More precisely we have
\begin{teo} \label{m-1-new}
Let $(M^{2m},g,J), m \geq 3$, be a connected  K\"ahler manifold. Then either: 
\begin{itemize}
\item[(i)]  $\HK^{0,m-1}(M,g) \cap \ker \di^{\star}=\{0\}$
\item[or]
\item[(ii)] at every point of $M$ the Ricci tensor of $g$ has at most two eigenvalues.
\end{itemize}
\end{teo}
\begin{proof}
Pick a form $\tau \neq 0$ in  $\HK^{0,m-1}(M,g) \cap \ker d^{\star}$. On the open subset of $M$ where $\tau$ does not vanish the distribution 
$\V:=\{X: X \lrcorner \tau=0\} \;$ has real dimension two and is $J$-invariant. 
By (iii) in Proposition \ref{prol-1} we get  $ \Ric X \,\lrcorner \,\, \tau=X_{0,1}\, \lrcorner \, \, C $ for all $X$ in $TM$, where $C:= \frac1m\bdel^{\star} \bdel \tau$ is in $\Lambda^{0,m-1}M$. Contracting again with a vector $V$ in $\V$ shows
that $C$ vanishes on $\V$ as well and thus $\V$ is preserved by $\Ric$. If $\H$ denotes the orthogonal complement of $\V$ in $TM$, both $C$ and $\tau$ belong to $\Lambda^{0,m-1}\H$ which has complex dimension one. 
Thus $C=f \tau$ for some function~$f$, i.e. $\Ric(X) \lrcorner \, \tau = f X_{0,1} \lrcorner \, \tau$.  It follows that $\Ric$ acts on $\H$ by  multiplication with a function, since $\Ric (\H) \subseteq \H$.
Now assume that the open set where the  Ricci tensor has at least three distinct eigenvalues is not empty. 
Then $\tau$ vanishes over this set  and therefor $\tau $ is vanishing everywhere by the argument above and
Corollary~\ref{zero}, thus  contradicting the assumption $\tau \neq 0$. 
\end{proof}
\subsection{Conformal Killing forms in middle dimension}\label{Cm}
According to \cite{AUT}[Definition 5.2] a real valued form $\varphi \in \Lambda^{(1,m-1)+(m-1,1)}M$ is, in the present conventions, a special $m$-form provided there exists a real valued 
form $\tau \in \Lambda^{(0,m-1)+(m-1,0)}M$ such that 
\begin{equation} \label{spec-real}
\nabla_U \varphi=(JU)^{\flat} \w \tau+U^{\flat} \w \bb{J}\tau+(U \lrcorner \, \tau) \w \omega
\end{equation}
for all $U \in TM$. Here the operator $\bb{J}$ acts according to $\bb{J}\tau=\tau(J \cdot, \ldots, \cdot)$. Note that the operator $J$ from \cite{AUT} acts on $\Lambda^{(0,m-1)}M \oplus \Lambda^{(m-1,0)}M$ as $-(m-1)\bb{J}$. Because the r.h.s. in \eqref{spec-real} is primitive so must be $\varphi$, up to parallel forms. Based on this we assume in what follows that $\varphi \in \Lambda_0^{(1,m-1)+(m-1,1)}M$.

When $m \geq 3$ we have a splitting 
$\Lambda_0^{(1,m-1)+(m-1,1)}M \otimes_{\bb{R}} \bb{C}=\Lambda_0^{(1,m-1)}M \oplus \Lambda^{(m-1,1)}_0M$ from which it is easy to deduce, by projection, 
that real-valued special $m$-forms are equivalently described by \eqref{special-i}. When $m=2$ equations \eqref{spec-real} and \eqref{special-i} are inequivalent. 

Therefore, in the rest of this article the aim is to describe K\"ahler manifolds $(M^{2m},g,J)$ of complex dimension $m \geq 3$ admitting pairs $(\varphi, \tau)$ of forms
satisfying \eqref{special-i}. We start with collecting a few immediate consequences of the
definition of special forms. By contraction of Equation~\eqref{special-i} any special form $\varphi$ satisfies 
\begin{equation} \label{eqn1}
\partial \varphi=0, \quad \overline{\partial} \varphi=\frac{i(m+1)}{2}L_{\omega} \tau. 
\end{equation}
Since $\varphi$ is primitive, using the K\"ahler identities it follows that 
\begin{equation} \label{eqn2}
\overline{\partial}^{\star} \varphi=0, \quad  \partial^{\star} \varphi=-\frac{m+1}{2} \tau \ .
\end{equation}
In particular $\tau$ is coclosed
and  the K\"ahler identities imply that $\del \tau$ is primitive.
Furthermore the dual K\"ahler identities yield 
\begin{equation} \label{eqn4}
\di^{\star}\!\di\! \varphi \;=\; \frac{m+1}{2}(\overline{\partial}\tau-\partial \tau) ,\quad
 \ \di\! \di^{\star}\!\varphi \;=\; -\frac{m+1}{2}(\bdel \tau+\partial \tau) ,
\end{equation} 
hence 
\begin{equation} \label{eqn5}
\Delta \varphi \;=\; -(m+1) \partial \tau. 
\end{equation}

Obviously any conformal Killing form is a Hermitian twistor form. Moreover it follows from \cite{AUT} that
primitive conformal Killing forms of type $(1,m-1)$ are precisely special forms as defined above.
To go on further the following algebraic lemma is needed.
\begin{lema} \label{ric-a} 
For any vector field $X$ and any form $\varphi$ in $\Lambda^{1,m-1}_0M$ we have 
$$
2(\Ric(X) \lrcorner \, \varphi)_{\Lambda^{0,m-1}M} \;=\; X_{1,0}\, \lrcorner \,\, (\Ric(\varphi) \;+\; i \, [\rho, \varphi])  \ .
$$
\end{lema}
\begin{proof}Let $\gamma: TM \to \Lambda^{0,m-1}M, \gamma(X)=(\Ric X \lrcorner \, \varphi)_{\Lambda^{0,m-1}M}=(\Ric X)_{1,0} \lrcorner \, \varphi$. Then 
$\gamma(JX)=i\gamma(X)$ for all $X$ in $TM$. Hence we can write $\gamma(X)=(X \lrcorner \, q)_{\Lambda^{0,m-1}M}$ for some $q$ in $\Lambda^{1,m-1}M$. After an elementary contraction $2q=\Ric(\varphi)+i[\rho, \varphi]$ and the claim is proved.
\end{proof}

The proposition below gives a relation  between conformal Killing forms in middle dimension and Hermitian Killing forms. We also  derive some of the additional consequences imposed by this set-up.

\begin{pro} \label{pro1}
Let $\varphi$  be a conformal Killing form in $\;\Lambda^{1,m-1}_0M$. The following hold
\begin{itemize}
\item[(i)]  $\tau =- \frac{2}{m+1} \del^*\varphi  \in   \HK^{0,m-1}(M,g) \cap \ker \di^{\star}$
\item[(ii)] $[\rho, \varphi]=i \, \partial \tau $
\item[(iii)] $\Ric(\partial \tau)=\frac{\s}{2} \del \tau $
\item[(iv)] we have 
\begin{equation} \label{ldt}
\Delta_{\bdel} \tau=\frac{m}{m-1}\Ric(\tau).
\end{equation}
\end{itemize}
\end{pro}
\begin{proof} 
Since $\varphi$ is a conformal Killing $m$-form it satisfies by (i) of Proposition \ref{der12}:
$$
 \frac{m}{m+1}\nabla_X(\di^{\star}\!\varphi)+\frac{1}{m+1}X \lrcorner \, \di^{\star}\!\di\varphi
 \;=\;
 \Ric X \lrcorner \, \varphi \;+\;  \frac{1}{2}( X \lrcorner \, \r \varphi-\r (X \lrcorner \, \varphi))  $$
for all vector fields $X$. Replacing therein the expressions of $\di^{\star}\! \varphi$ and $\di^{\star}\!\di\!\varphi$ as indicated in equations  \eqref{eqn1} and \eqref{eqn4} yields 
\begin{equation*}
m \nabla_X \tau=X \lrcorner \, (-\r \varphi+\bdel \tau-\partial \tau)-2 \Ric X \lrcorner \, \varphi+\r (X \lrcorner \, \varphi) 
\end{equation*} 
for all vector fields $X$. Recall now that the curvature tensor $R$ of the metric $g$ is a section of  
$ \Lambda^{1,1} \otimes \Lambda^{1,1}$, implying that 
$\r$ preserves the bi-type decomposition and that $\r (\Lambda^{0,\star})=0$. Therefore the components of the
equation above on $\Lambda^{0,m-1}M$ respectively $\Lambda^{1,m-2}M$ are 
\begin{equation*}
\begin{split}
m\nabla_X \tau=&(X \lrcorner \, (-\r \varphi+\overline{\partial} \tau-\partial \tau))_{\Lambda^{0,m-1}M}-2(\Ric X \lrcorner \, \varphi)_{\Lambda^{0,m-1}M}  \\
\end{split}
\end{equation*}
as well as 
\begin{equation} \label{proj2}
\begin{split}
0=&-(X \lrcorner \, (\r \varphi+\del \tau))_{\Lambda^{1,m-2}M}-2(\Ric X \lrcorner \, \varphi)_{\Lambda^{1,m-2}M}+\r(X \lrcorner \, \varphi) 
\end{split}
\end{equation}
for all $X$ in $TM$. Using \eqref{proj} the first equation simplifies to
\begin{equation} \label{eqn6}
m\nabla_X\tau=-X_{1,0} \lrcorner \, (\r \varphi+\partial \tau+\Ric(\varphi)+i [\rho, \varphi])+X \lrcorner \, \bdel \tau 
\end{equation}
for all $X$ in $TM$, after using Lemma \ref{ric-a} as well. \\
(i) follows now from the definitions and \eqref{eqn6}.\\
(ii) since $\nabla^{1,0}=\del$ on $\Lambda^{0,m-1}M$ it follows from \eqref{eqn6} that 
$$ \r \varphi+\partial \tau+\Ric(\varphi)+i[\rho,\varphi]=-m\partial \tau. 
$$
However, since $\varphi$ is a conformal Killing $m$-form we have by \eqref{lap} and \eqref{eqn5} that 
\begin{equation} \label{relfitau}
\r \varphi+\Ric(\varphi)=\frac{m}{m+1}\Delta \varphi=-m\partial \tau 
\end{equation}
and the claim follows. \\
(iii) by \eqref{obs} we have $\di\! \r\!\di\!\varphi=0$. Since $\di\!\varphi=\frac{i(m+1)}{2}L_{\omega} \tau$ by \eqref{eqn1} and $\r (\omega \wedge \tau)=-2 \rho \wedge \tau$ by \eqref{cc1}
we obtain that $\rho \wedge \di\!\tau=0$. Therefore $\rho \wedge \partial \tau=0$ by type considerations. Applying $L^{\star}_{\omega}$ while using \eqref{ls} proves the claim 
since $\partial \tau$ is primitive.\\
(iv) since $\di^{\star}\! \tau=0$ we have $\Delta_{\del}\tau=\del^{\star} \del \tau$ and $\Delta_{\bdel}\tau=\bdel^{\star} \bdel \tau$. The claim follows from the equality of the Laplacians $\Delta_{\del}$ 
and $\Delta_{\bdel}$ combined with (iv) in Proposition \ref{prol-1}.
\end{proof}

%

%
%
By combining Proposition~\ref{pro1},(i)  and Theorem~\ref{m-1-new}  we obtain
\begin{coro} \label{coroL22}
Let $(M^{2m}, g,J), m \geq 3$, be a K\"ahler manifold  admitting a non-parallel conformal Killing form in 
$\Lambda^{1,m-1}_0M$. Then at every point of $M$ the Ricci tensor of 
$g$ has at most two eigenvalues.
\end{coro}
We conclude this section with the following observation which will play a crucial role in obtaining geometric integrability conditions in section \ref{dif-str}.

\begin{lema} \label{sscal}
Let $(\varphi, \tau)$ solve \eqref{special-i}. Assuming that 
$\Ric(\varphi)=\frac{\s}{2}\varphi$
we have 
\begin{equation*}
J \grad (\s) \, \lrcorner \,\, \varphi \;=\; i \,\tfrac{m-3}{2} \, \s \, \tau\;+\; \, i\, \tfrac{2(2m-1)}{m-1}\, \Ric(\tau).
\end{equation*}
\end{lema}
\begin{proof}
We analyse the curvature-type constraint $\di( \mathfrak{R} \varphi)=\frac{m-1}{m+1}\, \mathfrak{R}(\di \varphi)$ which holds by Proposition \ref{der12},(ii)
since $\varphi$ is a conformal Killing $m$-form. Because 
$\Ric(\varphi)=\frac{\s}{2} \varphi$ by assumption and recalling that $\mathfrak{R} \varphi=- \Ric(\varphi)-m \del \tau$ (see \eqref{relfitau}), we obtain
$$ 
\mathfrak{R} \varphi \;=\; - \, \tfrac{\s}{2} \, \varphi-m \, \del \tau.
$$
We have  $\di\! \varphi=\frac{i}{2}(m+1)\, L_{\omega} \tau$ by \eqref{eqn1}, thus computing in both sides of the curvature constraint and  taking into account \eqref{cc1}, leads to 
\begin{equation} \label{temp-1}
\tfrac12\di \s \w \varphi \;+\; \tfrac{i}{4} \,  \, \s \, (m+1) \, L_{\omega}\tau +  m \, \bdel \del \tau
=\; -\,  \tfrac{i}{2} \, (m-1)\, \mathfrak{R}L_{\omega} \tau
=\;   i \, (m-1) \, L_\rho \,\tau \ .
\end{equation}
Now note that  \eqref{ls} and a simple computation, using that $\varphi$ is primitive, yield 
\begin{equation*}
\begin{split}
L^{\star}_{\omega} L_\rho \, \tau =  \tfrac{\s}{2} \,  \tau-\Ric(\tau) \qquad \mbox{and} \qquad  L^{\star}_{\omega}(\di \s \w \varphi)=
J \grad (\s) \, \lrcorner \, \,  \varphi \ .
\end{split}
\end{equation*}
Moreover from the K\"ahler identities for the co-closed form $\tau$ as well as \eqref{ldt} it follows that
$$ 
L^{\star}_{\omega}(\bdel \del \tau ) \;=\; -i \, \Delta_{\bdel} \tau=- \, i \, \tfrac{m}{m-1} \, \Ric(\tau)  \ .
$$
Gathering these facts proves the claim after applying $L^{\star}_{\omega}$
in \eqref{temp-1}, whilst taking into account that $L^{\star}_{\omega}L_{\omega}\tau=\tau$.
\end{proof}

\section{The Einstein set-up} \label{Einstein}
First consequences of Lemma \ref{sscal} include the full description of K\"ahler-Einstein metrics that admit special $m$-forms.
In what follows  we let $(M^{2m},g,J), m \geq 3$, be a connected K\"ahler-Einstein manifold equipped with a pair $(\varphi,\tau)$ solving \eqref{special-i}. We consider the length function $p=\vert \tau \vert^2$ and the vector field $K = -J\, \grad \,p$.
\begin{pro} \label{Ein-1} Assume that $\varphi$ is not parallel, i.e. $\tau \neq 0$. Then
\begin{itemize}
\item[(i)] $\Ric=0$ and $\del \tau=0$ 
\item[(ii)] $\nabla (\bar\del\tau)  = 0$, in particular  there exists a constant $\,k \geq 0\,$ such that $\;\vert \bdel \tau \vert ^2=k$ 
\item[(iii)] $\nabla K=-\frac{k}{m^2}J$, in particular $K \in \aut(M,g,J)$
\item[(iv)] we have $\vert \di\!p \vert^2=\frac{2k}{m^2}p$ and $\tau=\frac{m}{k} JK \lrcorner \, \bdel \tau$ when $k >0$
\item[(v)] we have $\li_{K}\! \tau=ik_1\tau$ with $k_1=\frac{k}{m} \geq 0$.
\end{itemize}
\end{pro}
\begin{proof}
(i) since the metric $g$ is Einstein, $\varphi \in \ker(\Ric-\frac{\s}{2})$ thus from Lemma~\ref{sscal} we obtain
$ 0 = \frac{m^2 + m -2}{2m} \, \s \, \tau$ due to $\grad(\s)=0$ and $\Ric(\tau) = \frac{m-1}{2m}\s \tau$.
It follows that  $\s \, \tau=0$, thus $g$ is Ricci flat, since $\tau$ does not vanish identically on $M$ by assumption. Moreover,
$\del \tau = 0$ by Proposition \ref{pro1},(ii).\\
(ii) the form $\bar\del\tau$ belongs to $\mathfrak{h}^{0,m}(M,g)$ by Proposition \ref{prol-1},(i), i.e. $\nabla^{0,1}\bar\del\tau = 0$. Since 
$\nabla^{1,0}=\del$ on $\Lambda^{0,\star}M$
and $\del\tau = 0 $ by (i), it follows that 
$\nabla^{1,0}\bar\del\tau = \del\bar\del \tau =  - \bar\del \del \tau =0$. Hence $\bar\del \tau$ is parallel and thus has constant length.\\
(iii) and (iv) are proved at the same time. Because $\del \tau=0$ the Hermitian Killing equation for $\tau$ reads $\nabla \tau=\frac{1}{m}\bdel \tau$. 
When $k=0$ it follows that $\tau$ is parallel hence $\di\! p=0$ and $K=0$ and both statements are trivially satisfied. \\
Assume thus that $k>0$. Then $\bdel \tau$ is a complex volume from, in particular it satisfies the algebraic identity 
\begin{equation} \label{vff}
\langle X_2\, \lrcorner \,\, \bdel \tau, X_1\, \lrcorner \,\, \bdel \tau \rangle + \langle X_1 \,\lrcorner \,\, \bdel \tau, X_2 \,\lrcorner \,\, \bdel \tau \rangle =k g(X_1,X_2) 
\end{equation}
for all $X_1,X_2$ in $TM$. Let the real valued vector field $\z$ be determined from $ \tau=\z \,\lrcorner \,\, \bdel \tau$. Because $\bdel \tau$ is parallel, the Hermitian-Killing equation for $\tau$ becomes $\nabla \z=\frac{1}{m} 1_{TM}$. It follows that 
\begin{equation*} 
\begin{split}
\di\!p(X)=&\langle \nabla_X \tau, \tau \rangle \,+\, \langle \tau, \nabla_X \tau \rangle=
\tfrac{1}{m} (\langle X \,\lrcorner \,\, \bdel \tau, \tau \rangle \,+\, \langle \tau, X\, \lrcorner \,\, \bdel \tau \rangle)\\
=&\tfrac{1}{m} (\langle X \,\lrcorner \, \bdel \tau, \z \lrcorner \, \bdel \tau \rangle + \langle \z \lrcorner \, \bdel \tau, X\, \lrcorner \,\, \bdel \tau \rangle)
\end{split}
\end{equation*}
for all vector fields $X$. Using \eqref{vff} leads to $\z=\frac{m}{k} \grad(p)$ and (iii) is proved. To compute the norm of $\di \!p$ we use the  
identity 
\begin{equation*}\label{norm}
p=\vert \tau \vert^2=\vert \z \lrcorner \,\, \bdel \tau \vert^2=\tfrac{k}{2}\vert \z \vert^2,
\end{equation*}
which is granted by the definition of $\z$ and \eqref{vff}.\\
(v) since $\tau \in \Lambda^{0,m-1}M$ we have $K \lrcorner \, \tau=iJK \lrcorner \, \tau$. At the same time $JK$ is proportional to $\zeta$ thus 
$K \lrcorner \, \tau=0$ due to $\tau=\zeta \lrcorner \, \bdel \tau$. Taking into account the vanishing of $\del \tau$ from (i) as well as (iv) leads to  
$ \li_{K}\! \tau =K \lrcorner \, \bdel \tau=\tfrac{ik}{m}\tau $.
\end{proof}
From parts (iii) and (iv) in Proposition \ref{Ein-1} it is easy to derive, for future use, that 
\begin{equation}\label{hess}
\mathrm{Hess}(p)=\tfrac{k}{m^2}g 
\end{equation}
as well as 
\begin{equation} \label{lap-e}
\frac{1}{2m}p\Delta p+\vert \di\!p\vert^2=\frac{k_1}{m}p.
\end{equation}
\begin{rema} \label{ein-tauf}
The distribution $\V:=\ker(\tau)$ is spanned by the vector fields $K$ and $JK$. Moreover $\V$ is obviously totally geodesic and the leaves
are flat 2-dimensional manifolds.
\end{rema}
Before proving the main result of this section we record the following simple observation.
\begin{lema}\label{critical}
Let $(M,g, J)$ and $(\varphi, \tau)$ be as above and assume $k>0$. Then $\tau$ has no zeroes and $p$ no critical points.
\end{lema}
\begin{proof}
By Proposition  \ref{Ein-1},(iv) the zeroes of $\tau$ are precisely the critical points of $p$. The latter set coincides with the zero set of the 
Killing vector field $K$; since $\ker(\nabla K)=0$ it follows that $p$ has only isolated critical points.
Assume that $x\in M$ is a critical point of $p$ and let $\gamma$ be the  integral curve of
$\grad(p)$ through the point $x=\gamma(0)$. Then the function $f(t) := p(\gamma(t))$ satisfies the differential equation 
$
\dot f(t)  = \frac{2k}{m^2}f(t)
$
with $f(0)=0$ which forces $f$ to be identically zero. Thus all points on $\gamma$ are critical points of $p$, which is  a contradiction since the
critical points of $p$ are isolated.
\end{proof}

To state our classification result we recall the following.
\begin{defi} \label{def-cone}
Let $(M,g,J)$ be K\"ahler. A cone structure for $g$ is a vector field $V_M$ such that $\nabla V_M=1_{TM}$. 
\end{defi}
A local model for this type of geometry is provided by the metric cone $M=P \times \mathbb{R}_+$ with cone metric $ g= r^2 g_P+ (\di\!r)^2$, where $(P,g_P)$ is a Sasaki  manifold. In this case the cone vector field is $V_M= r \frac{\partial}{\partial r}$. The metric cone is Einstein \cite{Baer} iff it is Ricci flat and $g_P$ is an Einstein metric. More adapted to K\"ahler geometry is thinking of $P$ as a local $\bb{R}$-bundle with K\"ahler base $(N,g_N,J_N)$. Then $(M,g,J)$ is oftenly referred to as the conification of the latter \cite{MaRo}. See also section \ref{ex} for putting this into the perspective Sasaki-Einsteinof the Calabi construction.  
\begin{teo} \label{E-thm}
Let $(M^{2m}, g, J), m \geq 3$ be a K\"ahler-Einstein manifold. There exists a solution $(\varphi,\tau)$ to \eqref{special-i} with 
$\bdel \tau$  not identically zero if and only if ,the following hold
\begin{itemize}
\item[(i)] $(M,g,J)$ admits a cone structure $V_M$ and a normalised complex volume form $\Psi_M$
\item[(ii)] up to rescaling we have 
\begin{equation} \label{soln1}
\varphi=\frac{1}{4} V_M^{1,0} \wedge (V_M \lrcorner \, \Psi_M)+\varphi_0, \ \tau=V_M \lrcorner \, \Psi_M
\end{equation}
where $\nabla \varphi_0=0$.
\end{itemize}
\end{teo}
\begin{proof}
(i) the function $p$ has no critical points by Lemma \ref{critical}. Results of \cite{Tash}(see \cite{Ku}, Lemma 12 for further details) ensure that around any point of $M$ there exists an open neighborhood on which  the
metric $g$ is an explicitly given warped product metric. Since $g$  is Ricci flat it follows that the warping function has to be linear, i.e.
$(M, g)$ is locally a metric cone. Moreover, since the manifold is Ricci flat and K\"ahler it has to be a cone over a Sasaki-Einstein manifold.\\
(ii) letting $V_M:=\frac{m^2}{k}JK$ and $\Psi_M:=\frac{1}{k}\bdel \tau$ yields, by Proposition \ref{Ein-1}, a cone vector field and a normalised complex volume form such that $\tau=\frac{k}{m}V_M \lrcorner \, \Psi_M$. Taking into account that $K_M=-JV_M \in \aut(M,g,J)$ we see that equation \eqref{special-i} is preserved by the Lie derivative $\li_{K_M}$, that is $(\li_{K_M} \varphi, \li_{K_M}\tau)$ is again a solution to \eqref{special-i}. Since $\li_{K_M} \tau=im \tau$ by (v) in Proposition \ref{Ein-1} it follows that $\li_{K_M} \varphi-im \varphi$ is parallel.
To compute the Lie derivative of $\varphi$ we recall the general formula 
\begin{equation} \label{lncom}
\li_{K_M}\varphi= \nabla_{K_M}\varphi - [\nabla K_M, \varphi]
\end{equation} 
where $[\cdot, \cdot]$ is the algebraic commutator introduced in
section \ref{CFK-s} which describes the standard action of skew-symmetric endomorphisms on forms.
Accordingly, taking into account Proposition~\ref{Ein-1},(iii),
the complex type of $\varphi$, as well as \eqref{special-i} with $K_M \lrcorner \, \tau = 0$ leads to
$$ 
\li_{K_M}\!\varphi-im \varphi=\nabla_{K_M} \varphi -2i\varphi=K^{1,0}_M \w \tau -2i\varphi.
$$ 
That is $\varphi=-\frac{i}{2}K^{1,0}_M \w \tau+\varphi_0$ where $\nabla \varphi_0=0$. 
\end{proof}
\begin{rema} \label{E-rmk2}
Assume $k > 0$. From results of Tashiro \cite{Tash} a complete Riemannian manifold $M$ admits a non-constant solution to \eqref{hess} if and only if it is  isometric (hence biholomorphic) to $\mathbb C^m$. Then $(\varphi, \tau)$ are described by part (ii) in Theorem \ref{E-thm} with $\Psi$ proportional to 
the standard complex volume form on $\bb{C}^m$. 
In particular the manifold $M$ cannot be compact. Alternatively this can be checked directly as follows:  first $\del \tau = 0$ implies, by using \eqref{eqn5}, that $\Delta \varphi=0$. Then integration yields $\di^{\star}\!\varphi=0$ thus
$\tau=0$ which contradicts the assumption $k>0$.
\end{rema}
There remains to treat instances when $\bdel \tau=0$. Then $\tau$ must 
be parallel by the Hermitian Killing equation. Moreover $\V$ is a parallel distribution thus we obtain a local splitting 
\begin{equation} \label{pe-split}
(M,g,J)=(\bb{C},g_{\bb{C}},J_{\bb{C}}) \times (N,g_N,J_N)
\end{equation}
where $(g_{\mathbb{C}},J_{\bb{C}})$ is the standard flat structure on $\bb{C}$ and $(N,g_N,J_N)$ is Ricci flat K\"ahler. Let $w$ be 
a complex co-ordinate on $\mathbb{C}$ with $\di\! w \in \Lambda^{0,1}\bb{C}$ and let $\omega_{\bb{C}}=\frac{1}{2i}\di\!w \wedge \di\! \overline{w}$.  We also fix a complex volume form $\Psi_N$ on $N$, normalised to 
$\vert \Psi_N \vert=1$.  
An explicit description of the(non complete) solutions is given below. As the computations are entirely similar to those above we only state the result and omit the details.
\begin{pro} \label{pro-E}
Let $(M^{2m},g,J),m\geq 3$, be K\"ahler-Einstein equipped with a pair $(\varphi, \tau)$ solving \eqref{special-i} and satisfying $\bdel \tau=0$. Assume that $\tau$ does not vanish identically. Then 
\begin{itemize}
\item[(i)] $(N,g_N,J_N)$ admits a cone structure $V_N$ and a normalised volume form $\Psi_N$
\item[(ii)] up to rescaling $\tau$ and adding parallel forms to $\varphi$ we have 
\begin{equation} \label{type2}
\varphi=i(\omega_{\bb{C}}-\omega_N) \wedge (V_N \lrcorner \, \Psi_N)+\frac{w}{2}\!\di\!\overline{w} \wedge \Psi_N, \quad \tau=\Psi_N.
\end{equation}
\end{itemize}
\end{pro}
Product metrics as in \eqref{pe-split} admit a second pair of non-trivial solutions to \eqref{special-i} given by 
\begin{equation} \label{type22}
\begin{split}
&\varphi=2iw(\omega_{\bb{C}}-\omega_N) \wedge (V_N \lrcorner \, \Psi_N)+w^2\di\!\overline{w} \wedge \Psi_N-\di\!w \wedge V_N^{1,0}\wedge 
(V_N \lrcorner \, \Psi_N)\\
&\tau=w\Psi_N-\di\!w \wedge (V_N \lrcorner \, \Psi_N).
\end{split}
\end{equation}
To check this it is enough to observe that the product $(M,g,J)$ has cone structure $V_{\bb{C}}+V_N$, where $V_{\bb{C}}$ is the cone structure of 
$(g_{\bb{C}},J_{\bb{C}})$ and normalised complex volume form 
$\di \!w \w \Psi_N$. That $(\varphi,\tau)$ given by \eqref{type22} satisfies \eqref{special-i} follows from Theorem \ref{E-thm} applied to this set of data.
\section{Classification of conformal Killing forms}
\subsection{Localisation and density} \label{loc-dens}
Pick a pair $(\varphi, \tau)$ solving \eqref{special-i}. To achieve the classification of such objects we will take advantage 
of the additional constraints this set-up imposes on the Hermitian Killing form $\tau$. We consider, over the open set 
\begin{equation*}
M_0:=\{x \in M :  \tau_x \neq 0\}\,
\end{equation*} the $2$-dimensional distribution $\V=\{U \in TM : U \lrcorner \, \tau=0\}$ together with its orthogonal complement $\H$ in $TM$, as introduced in the proof of \mbox{Theorem \ref{m-1-new}}. Both are complex, that is invariant under $J$, and it follows from
the proof of Theorem \ref{m-1-new} that the Ricci tensor of $g$ is diagonal with respect to the splitting 
\begin{equation} \label{split}
TM=\V \oplus \H \ .
\end{equation}
That is 
$$ 
\Ric_{\vert \V} \;=\; \lambda_1 1_{\V}, \quad   \Ric_{\vert \H} \;=\; \lambda_2 1_\H
$$ 
for some functions $\lambda_1$ and $\lambda_2$ on $M_0$. In particular the scalar curvature of the metric $g$ reads 
 $\frac{\s}{2}=\lambda_1+(m-1)\lambda_2$. Since $\tau$ belongs to $\Lambda^{0,m-1} \H$ we have 
$\Ric(\tau)=(m-1)\lambda_2 \tau $ hence \eqref{ldt} is now equivalent, over $M_0$, to  
\begin{equation} \label{srel-1} 
 m \, \lambda_2 \tau \;=\;  \bdel^{\star} \bdel \tau \; =\;  \Delta_{\bdel} \tau \ .
\end{equation}
Recall that $ \tau$ is co-closed, thus taking the co-differential in  \eqref{srel-1} gives 
$\grad (\lambda_2) \,\lrcorner \,\, \tau=0$, i.e.
\begin{equation} \label{dl}
\di\! \lambda_2 \in \Lambda^1 \V
\end{equation}
over $M_0$.  Now applying $\bdel$ in \eqref{srel-1}  yields 
$$
m\, \bdel(\lambda_2 \tau) \;=\; \Delta_{\bdel} \, \bdel \tau \;=\; \tfrac{\s}{2} \,  \bdel \tau \ ,
$$ 
by taking into account \eqref{hol-top} for the form $\bdel \tau \in \mathfrak{h}^{0,m}(M,g)$. Finally we find
\begin{equation} \label{param-1} 
m \, \bdel \lambda_2 \wedge \tau \;=\;  (\lambda_1 -\lambda_2) \, \bdel \tau 
\qquad \mbox{over} \quad M_0 \ .
\end{equation}
Localising further we consider the set of non-Einstein points in $M_0$
\begin{equation*}
M_1:=\{x \in M_0: (\Ric_0)_x \neq 0\}
\end{equation*}
where $\Ric_0$ denotes the tracefree part of the Ricci tensor. We assume that $M_1$ is not empty, i.e. that is $g$ is {\it{not}} an Einstein metric. From \eqref{param-1} we have 
\begin{equation} \label{sc-rles2} 
\begin{split}
& \bdel \tau= m \theta^{0,1} \w \tau
\end{split}
\end{equation}
over $M_1$, where 
\begin{equation} \label{eigf-th}
\theta= \frac{\di\! \lambda_2}{\lambda_1-\lambda_2}.
\end{equation}
Lastly, define 
$M_2:=\{x \in M_1: \theta_x \neq 0\}$.
\begin{lema} \label{densl2}
Assume that $\bdel \tau$ does 
not vanish identically in $M$. The sets $\{x \in M_1 : \lambda_2(x) \neq 0\}$ and $M_2$ are dense in $M_1$.
\end{lema}
\begin{proof}
Assume that $\lambda_2$ vanishes on some open set $U_0$ in $M_1$. However $\lambda_1 $ is nowhere vanishing in $U_0$ since $g$ is not Ricci flat in $M_1$. 
From \eqref{param-1} it follows that $\bdel \tau=0$ on $U_0$ thus on $M$ by unique continuation.   
This contradicts the assumption on $\tau$ thus $U_0$ must be empty and the first claim is proved. To prove the second we record that the vanishing of $\theta$ on some open piece of $M_1$ is equivalent to the vanishing of $\bdel \tau$ over $M$. This follows from  \eqref{sc-rles2} and the unique continuation property for $\bdel \tau$.  
\end{proof}
When $M$ is compact the assumption that $\tau$ be not holomorphic is essentially redundant.
\begin{lema} \label{tau-holm}
Assume that $M$ is compact and that $\bdel \tau=0$. Then $\tau=0$ and $\nabla \varphi=0$.
\end{lema}
\begin{proof}
From $\bdel \tau=0$ we get $\Delta_{\bdel} \tau=0$ since $\tau$ is co-closed. Then $\Delta_{\del}\tau=0$ hence after integration $\del \tau=0$ and further $\Delta \varphi=0$ by \eqref{eqn5}. It follows that $\di^{\star} \! \varphi=0$ and finally that 
$\tau=0$.
\end{proof}
\subsection{Integrability conditions} \label{dif-str}
We assume that $\varphi$ is not parallel and that $g$ is not an Einstein metric, in other words the set $M_1$ is assumed to be non-empty.
The first objective is to determine, over the open set $M_0$, the structure of the operators $\Ric$ and $ [\rho, \cdot] $ acting on $\Lambda^{1,m-1}_0M$.
Split the K\"ahler form $\omega=\omega^{\V}+\omega^\H$ according to \eqref{split}. Purely algebraic considerations show that 
\begin{equation} \label{spliLa}
\Lambda^{1,m-1}_0M=(\omega^{\V}-\omega^{\H}) \w \Lambda^{0,m-2}\H \oplus (\Lambda^{1,0}\V \w \Lambda^{0,m-1}\H) \oplus 
(\Lambda^{0,1}\V \w \Lambda^{1,m-2}_0\H)
\end{equation}
over $M_0$.
Direct calculation taking into account only the structure of the Ricci tensor leads to the following 
\begin{lema} \label{prim}
Let $\psi=(\omega^{\V}-\omega^{\H}) \w \psi_1+\psi_2+\psi_3 \in \Lambda^{1,m-1}_0M_0$ be split according to \eqref{spliLa}. We have
\begin{itemize}
\item[(i)]  $\Ric(\psi)  = \;  \frac{\s}{2} \, \psi \;+\; (\lambda_1-\lambda_2) \, \omega \wedge \psi_1 $ 
\item[(ii)]  
$[\rho, \psi]  =i\lambda_2(m-2)(\omega^{\V}-\omega^\H) \wedge \psi_1+i \, ((m-1) \lambda_2 - \lambda_1) \, \psi_2 \; +\; i\, (\lambda_1
+ (m-3)\lambda_2)\psi_3.$
\end{itemize}
\end{lema}
In the following we will indicate with $\psi_k, 1 \leq k \leq 3$ the components of a generic form $\psi \in \Lambda^{1,m-1}_0M_0$ split as in Lemma \ref{prim}. The first set of integrability conditions on pairs $(\varphi, \tau)$ solving the special conformal Killing equation is obtained below. 
\begin{pro} \label{step2}
Let the pair $(\varphi, \tau)$ solve \eqref{special-i} and assume that $\bdel \tau$ does not vanish identically on 
$M$. Then 
\begin{itemize}
\item[(i)] we have $\varphi_1=(\del \tau)_1=0$ on $M_1$
\item[(ii)] $(\frac{\s}{2}-2\lambda_1)\, \varphi_2=(\del \tau)_2$ \;and\; $(\frac{\s}{2}-2\lambda_2)\, \varphi_3=(\del \tau)_3$ \; on $M_1$
\item[(iii)] $ \Ric (\varphi) = \; \frac{\s}{2} \, \varphi$ \; on $M$.
\end{itemize}

\end{pro} 
\begin{proof}
(i) and (ii) are proved at the same time. The form $\del \tau$ is primitive of type $(1,m-1)$ and belongs to $\ker(\Ric-\frac{\s}{2})$ according to  Proposition~\ref{pro1},(iii). By 
Lemma \ref{prim},(i) this entails the vanishing of $(\del \tau)_1$,  since $\lambda_1 \neq \lambda_2$ in $M_1$.
To deal with the components in $\varphi$ expand the relation $[\rho, \varphi] = i \del \tau$ from Proposition \ref{pro1},(iii) with the aid of (ii) in Lemma \ref{prim}. We get 
$$
 \lambda_2 \, \varphi_1 \; =\; 0, \qquad (\tfrac{\s}{2}-2\lambda_1)\varphi_2=(\del \tau)_2 , \qquad
 (\tfrac{\s}{2}-2\lambda_2)\varphi_3=(\del \tau)_3.
$$
This proves (ii) as well as the vanishing of $\varphi_1$ in $\{x \in M_1:\lambda_2(x)\neq 0 \}$. As the latter set is dense in $M_1$ by Lemma \ref{densl2} it follows that $\varphi_1=0$ on $M_1$.\\ 
(iii) holds in $M_1$ because $\varphi_1$ vanishes on the this set. On every open piece of $M_0$ where $g$ is Einstein the relation $\varphi \in \Ker(\Ric-\frac{\s}{2})$ is trivially satisfied. Thus (iii) holds in $M_0$ 
by a density argument. We conclude by taking into account that $M_0$ is dense in $M$.
\end{proof}
The final set of integrability conditions is derived below. 
\begin{pro} \label{c-fol} Assume that  $\bdel \tau$ does not vanish identically on 
$M$. The following hold over $M_1$
\begin{itemize}
\item[(i)] $\grad(\s)$ belongs to $ \V$
\item[(ii)] we have $\varphi_3=(\del \tau)_3=0$.
\end{itemize}
\end{pro}
\begin{proof}
Since $\varphi \in \ker(\Ric-\frac{\s}{2})$ by (iv) in Proposition \ref{step2} we can apply Lemma \ref{sscal}. The constraint therein reads, after taking into account that $\Ric(\tau)=\lambda_2(m-1) \tau$ in $M_1$
\begin{equation} \label{grad-scal}
J \grad (\s) \lrcorner \,  \, \varphi=i(c_1 \s +c_2 \lambda_2) \tau
\end{equation}
where $c_1=\frac{m-3}{2}$ and $c_2=2(2m-1)$. Letting $\xi:=J\grad (\s)$ we split $\xi=\xi_{\V}+\xi_{\H}$ according to $TM=\V \oplus \H$. In order to project \eqref{grad-scal} onto the splitting \eqref{spliLa} record that 
\begin{equation*}
\begin{split}
&\xi_{\V} \lrcorner \, \varphi_2 \in \Lambda^{0,m-1}\H, \quad    \xi_{\V} \lrcorner \, \varphi_3 \in \Lambda^{1,m-2}_0\H\\
&\xi_{\H} \lrcorner \, \varphi_2 \in \Lambda^{1,0}\V \w \Lambda^{m-2}\H, \quad \xi_{\H} \lrcorner \, \varphi_3 \in \Lambda^{0,1}\V \w \Lambda^{m-2}\H
\end{split}
\end{equation*}
directly from $\varphi_2 \in \Lambda^{1,0}\V \w \Lambda^{0,m-1}\H$ and $\varphi_3 \in \Lambda^{0,1}\V \w \Lambda^{1,m-2}_0\H$. These type considerations together with having $\tau \in \Lambda^{0,m}\H$ and $\varphi_1=0$ make that \eqref{grad-scal} is equivalent to 
\begin{equation} \label{gs1}
\xi_{\V} \lrcorner \, \varphi_3=0, \quad  \xi_{\H} \lrcorner \, \varphi_3=0
\end{equation} 
coupled with 
\begin{equation} \label{gs2}
\xi_{\V} \lrcorner \, \varphi_2=i(c_1 \s +c_2 \lambda_2) \tau, \quad  \xi_{\H} \lrcorner \, \varphi_2=0.
\end{equation} 
(i) Since $\dim_{\mathbb{C}}\H=m-1$ we have $\dim_{\mathbb{C}}\Lambda^{0,m-1}\H=1$. Consequently over the open set where $\varphi_2 \neq 0$ 
the space $\{X \in \H : X \lrcorner \, \varphi_2=0 \}$ vanishes so
$\grad( \s )\in \V$ by the second equation in \eqref{gs2}. On any open subset $\mathcal{O} \subseteq M_1$ where 
$\varphi_2=0$ we must have $c_1 \s +c_2 \lambda_2=0$ by the first equation in \eqref{gs2}. But $\grad(\lambda_2)$ belongs to $\V$ by \eqref{dl} thus 
$\grad(\s) \in \V$ as well in $\mathcal{O}$, provided that $c_1 \neq 0$. If $c_1=0$ we get $\lambda_2=0$ in $\mathcal{O}$ hence the latter must be empty 
by Lemma \ref{densl2}.
We have proved, by density, that $\grad (\s)$ belongs to $ \V$ over $M^{1}$.\\
(ii) Assume that $\varphi_3 \neq 0$ at some point and therefore in some open subset $U_0 \subseteq M_1$; then $\grad( \s)=0$ in $U_0$ by the first equation in \eqref{gs1} and (i). It follows from the first equation in \eqref{gs2}
that $c_1 \s +c_2 \lambda_2=0$, in particular $\di\! \lambda_2=0$ in $U_0$ since $c_2 \neq 0$. This is a contradiction with having $M_2$ dense 
in $M_1$  (see Lemma \ref{densl2}).
Thus $\varphi_3=0$ and the vanishing of $(\del \tau)_3$ follows now from (ii) in Proposition \ref{step2}.\\
\end{proof}
\subsection{Calabi-type metrics} \label{CAL}
Given a K\"ahler manifold $(M,g,J)$ recall that a foliation 
$\mathcal{F}$ with leaf tangent distribution $\D \subseteq TM$ is called complex if $J\D=\D$. A complex foliation is called holomorphic if it is locally spanned by holomorphic vector fields. Equivalently $(\li_{V}\!J)TM \subseteq \D$ for all 
$V \in \D$. According to \cite{V}, the foliation $\mathcal{F}$ is called conformal if $\li_V\!g=\theta(V) g$ on $\D^{\perp}$ for all $V$ in $\D$, where $\theta$ is in 
$\Lambda^1\D$. It is usually referred to as the Lee form of the foliation. If moreover $\di\!\theta=0$ respectively $\theta$ is exact then $\F$ is called homothetic respectively globally homothetic. Following \cite{NaOr} we present below an easy exterior differential criterion to show that a given foliation by complex curves is conformal.
\begin{pro} \label{crit}Let $(M^{2m},g,J)$ be K\"ahler equipped with a holomorphic foliation $\F$. Let $\V$ be the   leaf tangent distribution, $\H:=\V^{\perp}$ and split the 
K\"ahler form $\omega=\omega^{\V}+\omega^{\H}$ according to $TM=\V \oplus \H$. The following statements are equivalent
\begin{itemize}
\item[(i)] $\F$ is conformal with Lee form $\theta$
\item[(ii)] we have $\di\!\omega^{\V}=-\theta \wedge \omega^{\H}$ for some $\theta \in \Lambda^1\V$.
\end{itemize}
\end{pro}
As a direct consequence we have 
\begin{coro} \label{conf-ho}
Let $(M^{2m},g,J)$ be K\"ahler equipped with a holomorphic and conformal foliation $\F$ of complex codimension $\geq 2$. Then $\F$ is homothetic.
\end{coro}
See \cite{NaOr} for more details.
The full local description of K\"ahler metrics that admit homothetic foliations by complex curves has been obtained in \cite{Chiossi-Nagy}. Of relevance here is the following particular instance of that setup.
\begin{defi} \label{C-T}
Let $(M^{2m},g,J),m \geq 2$, be K\"ahler. It is called of Calabi type provided it admits a foliation $\mathcal{F}$ by complex curves which is totally geodesic and homothetic. 
\end{defi}
By \cite{Chiossi-Nagy} Calabi-type metrics correspond locally to K\"ahler metrics 
obtained by the Calabi Ansatz \cite{C, Hw-S}, on total spaces of holomorphic line bundles. 
\subsection{The canonical foliation} \label{sec-fol}
Below we show by using the integrability conditions $(\del \tau)_1=(\del \tau)_3=0$ from the previous section that the conformal Killing equation \eqref{special-i} forces the K\"ahler structure $(g,J)$ to be of Calabi type over $M_1$. First direct consequences include an explicit expression for the form $\varphi$. At the same time, based on Proposition \ref{c-fol}  we describe a geometric way to fully reduce the integration of equation \eqref{special-i}  on $(\varphi,\tau)$ to solving a first order ODE. 
\begin{pro} \label{fol-int-new}Assume that $(\varphi,\tau)$ solves \eqref{special-i} and that $\bdel \tau$ does not vanish identically in $M$. The K\"ahler manifold $(M_1,g,J)$ is of Calabi type with respect to 
$\V=\ker (\tau)$.
The Lee form $\theta$ of the corresponding foliation is given by \eqref{eigf-th}.

\end{pro}
\begin{proof} 
A short computation taking into account the vanishing of $\tau$ on $\V$ leads to the identity $(\li_V\!J)U \lrcorner \, \tau=2iV_{01} \lrcorner \, \nabla^{10}_U\tau$ whenever $V \in \V$ and $U \in TM$. But $\nabla^{10}\tau=\del \tau$ by the Hermitian-Killing equation and $\del \tau \in \Lambda^{1,0}\V \w \Lambda^{0,m-1}\H$ thus $V_{01} \lrcorner \, \del \tau=0$. We have showed that $\V$ is holomorphic. To show that $\V$ is totally geodesic take into account that $\tau$ vanishes on $\V$ together with the Hermitian-Killing equation on $\tau$ to arrive at 
$\nabla_VW \lrcorner \, \tau=-W_{01} \lrcorner \, (V_{10} \lrcorner \, \del \tau+V_{01} \lrcorner \, \bdel \tau)$, for all $V,W \in \V$. The claim follows from 
having $\dim_{\mathbb{C}}\V=1$ and $(\del \tau)_1=0$. \\
Because $\rho=\lambda_1 \omega^{\V}+\lambda_2 \omega^{\H}$ and $\omega=\omega^{\V}+\omega^{\H}$ are closed differentiation yields 
\begin{equation*} 
(\lambda_1-\lambda_2) \di\! \omega^{\V}+\di\! \lambda_1 \wedge \omega^{\V}+\di\! \lambda_2 \wedge \omega^{\H}=0.
\end{equation*}
But $\di\! \lambda_k \w \omega^{\V}=0, k=1,2$ by combining \eqref{dl} and (i) in Proposition \ref{c-fol}. In particular 
$\di\! \omega^{\V}=-\theta \w \omega^{\H}$
thus, since $\theta \in \Lambda^1\V$, the foliation induced by $\V$ is conformal with Lee form $\theta$ by Proposition \ref{crit}. That $\di\! \theta=0$ follows by Corollary \ref{conf-ho} whilst taking into account that $\dim_{\mathbb{C}}\H=m-1\geq 2$.
\end{proof}
The full local description of K\"ahler metrics admitting a conformal Killing form as in \eqref{special-i} is now at hand. We will use the specific geometric structure structure of the Calabi type  manifold $(M_1,g,J)$ as follows. It is a crucial observation that metrics of Calabi type admit a canonically defined symmetry, constructed as follows.
Let the vector field $\z$ be determined from $\theta=g(\z,\cdot)$. Around each point $x$ in $M_2$ consider some open set such that $\theta=\di\! f$ and denote $z=e^f$. By \cite{Chiossi-Nagy} we know that 
$$K=-zJ\z$$ is a holomorphic Killing vector field with moment map $z>0$, since $K \lrcorner \, \ \omega=\di \! z$. Again by \cite{Chiossi-Nagy} we know that 
$\di\! g(K,K) \w \di\! z=0$, that is $K$ is rigid in the terminology of \cite{ACG}. Hence, by localising further if necessary 
one can find an open connected subset $\mathcal{O}$ in $M_2$ around $x$, where $g(K,K)=\mathbb{X}(z)$ for some smooth map $\bb{X}:(0,\infty) \to (0, \infty)$; this is sometimes refered to as the momentum profile of $K$. Since $K$ is a holomorphic Killing vector field we have the general formula $-\di\! K^{\flat}=(L_{JK}g)(J \cdot, \cdot)$ where the metric dual $K^{\flat}=g(K, \cdot)$. By taking into account that  $\V$ is totally geodesic and conformal with Lee form $\theta=\di \ln z$  this leads to 
\begin{equation} \label{der-K}
-\di\! K^{\flat}=\bb{X}^{\prime}(z)\omega^{\V}+z^{-1}\bb{X}(z)\omega^{\H}
\end{equation}
over $\mathcal{O}$. 
In what follows statements will be said to hold {\it{locally}} in $M_2$, provided  this happens on any connected open subset $\mathcal{O}$ of $M_2$ as introduced before. This type of further localisation is needed because we will work around points in $M_2$ where both
$K$ and the momentum profile $\bb{X}$ are well defined. 

Returning to the description of solutions of \eqref{special-i} we make the following technical
\begin{lema} \label{Liet} Assume that $\li_K\!\tau=i\chi \tau$ in $M_2$, for some real valued function $\chi:M_2 \to \mathbb{R}$ with $\di\! \chi \in \Lambda^1\V$. Then $\di\!\chi=0$.
\end{lema}
\begin{proof}
Apply the Lie derivative $\li_K$ in $\bdel \tau=m \theta^{01} \w \tau$.
Taking into account that $K$ is holomorphic and $L_K\theta=0$ yields $\bdel(\li_K\!\tau)=m \theta^{01} \w \li_K\! \tau$. Expansion using that $\li_K\!\tau=i\chi \tau$ 
and again the expression for $\bdel \tau$ leads to $\bdel \chi \w \tau=0$.  Because $\bdel \chi \in \Lambda^{0,1}\V$ and $\tau$ does not vanish 
in $M_2$ we get $\bdel \chi=0$. Since $\chi$ is real valued the claim follows.
\end{proof}
Locally the Calabi type manifold $M_1$ can be thought of as the total space of a conformal and holomorphic submersion $\pi:(M_1,g,J) \to (N^{2(m-1)},g_N,J_N)$, where $(N,g_N,J_N)$ is a K\"ahler manifold. This submersion has dilation factor $\sqrt{z}$, for $z\pi^{\star}g_N=g_{\vert \H}$ and totally geodesic fibres. 
\begin{pro}\label{loc-descr}\ Assume that $(\varphi,\tau)$ is a solution to \eqref{special-i} such that $\bdel \tau$ does not vanish identically in $M$. The following holds locally in $M_2$
\begin{itemize}
\item[(i)] we have $\varphi=\frac{z}{\bb{X}(z)} \del z \w \tau$
\item[(ii)] $\!\li_K\tau=ik\tau$ for some real constant $k$
\item[(iii)] $\bb{X}(z) =z(C_1z^{m}+\frac{2k}{m})$ for some $C_1 \in \mathbb{R}^{\times}$
\item[(iv)] the (local) base metric $g_N$ is Einstein with $\s^N=2(m-1)k$ 
\item[(v)] the eigenfunctions of the Ricci tensor are 
$$ \lambda_1=-m^2C_1z^{m-1}, \ \lambda_2=-mC_1z^{m-1} \ 
$$
in particular $\s=-2m(2m-1)C_1z^{m-1}$.
\end{itemize}
\end{pro}
\begin{proof}
Because $\varphi_1=\varphi_3=0$ respectively $(\del \tau)_1=(\del \tau)_3=0$ the forms $\varphi$ respectively $\del \tau$ belong to $\Lambda^{1,0}\V \w \Lambda^{0,m-1}\H$. The space $\Lambda^{0,m-1}\H$ is spanned by $\tau$ over $M_1$ since $\tau$ does not vanish in the latter set.
Thus, by also taking into account part (ii) in Proposition \ref{step2}
we can parametrise 
\begin{equation} \label{par-pair}
\varphi=c \w \tau, \quad \del \tau=(\frac{\s}{2}-2\lambda_1)c \w \tau
\end{equation}
for some $c \in \Lambda^{1,0}\V$. The parametrisation of $\del \tau$ above together with \eqref{sc-rles2} makes that the Hermitian Killing equation \eqref{HK} on $\tau$ in direction of $\V$ respectively $\H$ reads
\begin{equation} \label{hk-e0}
\nabla_V \tau=((\frac{\s}{2}-2\lambda_1)c+\theta^{0,1})(V) \, \tau 
\end{equation}
for all $V$ in $\V$, respectively 
\begin{equation} \label{hk-e2}
\nabla_X \tau=-\theta^{0,1} \w X \lrcorner \, \,  \tau 
\end{equation}
for all $X$ in $\H$, since $\theta \in \Lambda^1\V$ by \eqref{dl}. During the rest of the proof we will work on some connected open subset $\mathcal{O}$ of $M_2$ where both $K$ and $\bb{X}$ are well defined.\\
(i) we examine the various components of the conformal Killing equation \eqref{special-i} w.r.t. the parametrisation of $(\varphi,\tau)$ in \eqref{par-pair}. For convenience recall that 
$$ \nabla_U \varphi=U^{1,0} \w \tau+\frac{i}{2} \omega \wedge (U \lrcorner \, \tau)
$$
for all $U$ in $TM$. With $X \in \H$ this implies
\begin{equation} \label{phi1}
\begin{split}
\nabla_Xc \w \tau+c \w \nabla_X \tau\; =&\;\, X^{1,0} \w \tau+\frac{i}{2} \omega^{\H} \w (X \lrcorner \, \tau)+\frac{i}{2} \omega^{\V} \w (X \lrcorner \, \tau) \\
\;=&\; \,\frac{1}{2}X^{1,0} \w \tau+\frac{i}{2} \omega^{\V} \w (X \lrcorner \, \tau).
\end{split}
\end{equation}
We have used that $X^{1,0} \w \tau=-i \omega^{\H} \w (X \lrcorner \, \tau)$, a consequence, essentially, of having $\H$ of dimension 
$2(m-1)$. Projecting \eqref{phi1} onto $\Lambda^2\V \w \Lambda^{0,m-2}\H$ whilst taking \eqref{hk-e2} into account yields $-c \w \theta^{0,1}=\frac{i}{2} \omega^{\V}$ hence 
\begin{equation} \label{T1}
c=\vert \theta \vert^{-2} \theta^{1,0}=\frac{z}{\bb{X}}\del z
\end{equation} 
and the claim is proved. The remainder of \eqref{phi1} is equivalent with $\nabla_Xc=c(\z)X^{1,0}$. A short computation based on \eqref{T1} and \eqref{der-K} shows this is an identity.\\
(ii) from \eqref{sc-rles2} and \eqref{eigf-th} we get $ \di\! \tau=((\frac{\s}{2}-2\lambda_1)c+m\theta^{01}) \w \tau.$
Taking into account that $\theta(K)=0$ as well as \eqref{T1} it follows that 
$\li_K\!\tau=i \chi \tau$ where the real valued function $2\chi=\frac{m\bb{X}}{z}-z(\frac{\s}{2}-2\lambda_1)$. Due to (i) and (iii) in Proposition \ref{c-fol} this satisfies $\di\!\chi \in \Lambda^1\V$ thus Lemma \ref{Liet} ensures that $\chi=k$ on $\mathcal{O}$ for some $k \in \mathbb{R}$ and the claim follows. Moreover from the explicit expression for the function $\chi$ above  
\begin{equation} \label{s-pp}
z(\frac{\s}{2}-2\lambda_1)=\frac{m\bb{X}}{z}-2k.
\end{equation}
 \\
(iii) for $V \in \V$ the conformal Killing equation reads
$ \nabla_Vc \w \tau+c \w \nabla_V \tau=V^{1,0} \w \tau.$
Since $\V$ is totally geodesic $ \nabla_Vc$ belongs to $\Lambda^{1,0}\V$. Using \eqref{hk-e0} leads to 
\begin{equation} \label{e1}
\nabla_Vc+((\frac{\s}{2}-2\lambda_1)c+\theta^{01})(V)c=V^{10}.
\end{equation} 
By \eqref{T1} we know 
that $c=-\frac{iz}{\bb{X}(z)}g(K_{01}, \cdot)$; since $\nabla_VK_{01}=
\frac{i\bb{X}^{\prime}(z)}{2}V_{01}$ by \eqref{der-K}, we obtain  $\nabla_Vc=\frac{z\bb{X}^{\prime}(z)}{2\bb{X}}g(V_{01}, \cdot)-i(\frac{z}{\bb{X}})^{\prime}
\di\! z(V) g(K_{01}, \cdot).$
Plugging these together with $\theta^{0,1}=iz^{-1}K^{0,1}$ and \eqref{s-pp} in \eqref{e1} we see it is equivalent to the first order ODE
$\bb{X}^{\prime}=(m+1)\frac{\bb{X}}{z}-2k$. The claim follows by integration of the latter.\\
(iv) and (v) follow by using succesively (iii), \eqref{s-pp}, $\rho=\lambda_1\omega^{\V}+\lambda_2\omega^{\H}$ and the general formula for the Ricci form of a Calabi-type metric \cite{Tonn}. 
In our conventions it reads
\begin{equation*}
\rho=\pi^{\star}\rho^N+\Lambda_1\omega^{\V}+\Lambda_2\omega^{\H} 
\end{equation*}
where 
$\Lambda_1 = - \frac12(\bb{X}^{\prime \prime}(z) + \frac{m-1}{z^2}(\bb{X}^{\prime}(z)z - \bb{X}(z)))$
and
$\Lambda_2 = -\frac{1}{2z}(\bb{X}^{\prime}(z) + (m-1)\frac{\bb{X}(z)}{z}).$\\
\end{proof}
In other words in equation \eqref{special-i} the form $\varphi$ is entirely determined by $\tau$ and the moment map $z$ as well as the specific properties of the Calabi-type structure, i.e. Einstein base and explicit form of the momentum profile $\bb{X}$. Note that all the constraints involving $\theta$, including \eqref{eigf-th} are now solved.
Theorem \ref{main1} is thus proved up to determining the structure of $\tau$ w.r.t. this set of data. This 
is done in section \ref{lifts}.  Specifically we describe Hermitian Killing forms on Calabi type manifolds. As a consequence of 
Proposition \ref{ex-cal}, (ii), we obtain the local form of $\tau$ as given in Theorem \ref{main1}.

\subsection{Global classification} \label{g-clas}
Proposition \ref{loc-descr} leads naturally to the identification of the global invariants relevant to our set-up. The unique continuation property for Hermitian-Killing 
forms can be used to fix the constants of integration over the manifold. We continue working on a K\"ahler manifold $(M^{2m},g,J), m \geq 3$, equipped 
with a pair $(\varphi, \tau)$ satisfying equation \eqref{special-i}. Keeping the notation from the previous section we first prove that 
\begin{lema} \label{lengths}
We have locally in $M_2$ 
$$\vert \tau \vert^2 = C_2\bb{X}, \ \quad  \vert \varphi \vert^2=\frac{C_2}{2}z^2$$
where the constant $C_2$ is positive. 
\end{lema}
\begin{proof}
Taking the scalar product with $\tau$ in \eqref{hk-e0} yields $\di\! \ln \vert \tau \vert^2=(\frac{\s}{2}-2\lambda_1)(c+\overline{c})+\theta$.
By taking succesively into account \eqref{T1}, \eqref{s-pp}, that $\theta=\di\!\ln z$ together with the explicit expression for 
$\bb{X}$ in Proposition \ref{loc-descr},(iii) leads after some calculations to $\di\! \ln (\vert \tau \vert^2\bb{X}^{-1})=0$ and the claim on the norm of $\tau$ follows. The norm of $\varphi$ is computed with the aid of (i) in Proposition \ref{loc-descr}.
\end{proof}

\begin{pro} \label{glob-1}Let $(M^{2m},g,J), m \geq 3$, be a connected K\"ahler manifold equipped with a pair $(\varphi, \tau)$ satisfying \eqref{special-i}. Assume that $\bdel \tau$ is not identically zero on $M$. Then 
\begin{itemize}
\item[(i)] $p:=\vert \tau \vert^2+\frac{\s \vert \varphi \vert^2}{m(2m-1)}$ is a holomorphic Killing potential, that is 
$K_1 :=-J \grad \, p$ is an holomorphic Killing vector field
\item[(ii)] if $g$ is not Einstein there exists a constant $k_1 \geq 0$ such that $\li_{K_1}\tau=ik_1 \tau$ and $\li_{K_1}\!\varphi=ik_1 \varphi$. Moreover 
\begin{equation} \label{max}
\vert \di\! p \vert^2+\frac{1}{2m}p\Delta p \, = \, \frac{k_1}{m}p \, .
\end{equation}
\end{itemize}
\end{pro}
\begin{proof}
(i) Assume that $M_1$ is not empty,  i.e. $g$ is not Einstein. Locally in $M_2$ we have that 
$p=\frac{2kC_2}{m}z$ by Lemma \ref{lengths} and the expression for the scalar curvature  from Proposition \ref{loc-descr}. Hence $K_1=-J \grad \,  p=\frac{2C_2k}{m}K$ is a holomorphic Killing vector field w.r.t. $g$, around 
any point in $M_2$. Therefore $K_1$ has the required properties over $M_1$, since $M_2$ is dense in the latter. On any open set where the metric is Einstein, that is $\Ric=0$, we have $p=\vert \tau \vert^2$. Thus $K_1$ has again the required properties by part (iii) in Proposition \ref{Ein-1}. By density it follows that $K_1$ is a holomorphic Killing vector field in $M_0$. Since this set is dense in $M$ the claim is proved.\\
(ii) Pick some open set $\mathcal{O} \subseteq M_2$ where both $K$ and $\bb{X}$ are defined. Over $\mathcal{O}$ we let $k_1=\frac{2k^2C_2}{m}$. Since $K_1$ is a global Killing vector field 
$(\li_{K_1}-ik_1) \tau$ is a co-closed Hermitian Killing form. By Proposition \ref{loc-descr}, (ii), the form  $(\li_{K_1}-ik_1) \tau=0$ vanishes over $\mathcal{O}$ hence identically in $M$, as it follows from Corollary \ref{zero}.
Applying the operator $\li_{\! K_1}-ik_1$ in equation \eqref{special-i} it follows that 
$(\li_{\! K_1}-ik_1) \varphi$ is parallel. However, from (i) and (ii) in Proposition \ref{loc-descr} together with having $K_1$ holomorphic we get that 
$(\li_{\! K_1}-ik_1)\varphi=0$ over $\mathcal{O}$.  
These facts enable us to conclude that $(\li_{K_1}-ik_1) \varphi=0$ over $M$.

To establish \eqref{max} we first work locally in $M_2$. From \eqref{der-K} we have 
$-\langle \di\! K^{\flat}, \omega\rangle=\bb{X}^{\prime}(z) \vert \omega^{\V} \vert^2+\frac{\bb{X}}{z}\vert \omega^{\H} \vert^2=
-2k+2m \frac{\bb{X}}{z}$. But $K^{\flat}=J\di\!z$ hence 
$\langle \di\!K^{\flat}, \omega\rangle=\Delta z$ and further 
$\Delta z=2k-2m \frac{\bb{X}}{z}$. Since $p=\frac{2kC_2}{m}z$ we get $\vert\! \di\! p \vert^2=
(\frac{2kC_2}{m})^2\bb{X}(z)$ and it follows, by direct verification, that  equation \eqref{max} is satisfied around each point in $M_2$, hence by density in $M_1$. 

Let us now work over some open set in $M_0$ where 
$g$ is Einstein. Since $K_1=K$ in the notation of section \ref{Einstein} and $k=\vert \bdel \tau \vert^2$ we obtain, by Proposition \ref{Ein-1},(v)  that $\li_{K_1}\!\tau=ik_1 \tau$ with $k_1=\frac{k}{m} \geq 0$. Equation \eqref{max} is precisely equation \eqref{lap-e} established in section \ref{Einstein}.
Thus \eqref{max} holds 
around points where $g$ is Einstein, hence everywhere in $M_0$, respectively $M$, by density.
\end{proof}
$\\$
{\bf{Proof of Theorem \ref{main2}}.}
Recall that $M$ is assumed  to be compact and equipped with a primitive conformal Killing form, i.e. a pair
$(\varphi, \tau)$ satisfying \eqref{special-i}. If $\bdel \tau=0$ over $M$ we have $\tau=0$ by  Lemma \ref{tau-holm}. If $g$ is an Einstein metric then 
$\tau=0$ by  Remark \ref{ein-tauf}, (ii).

Therefore, we assume that the set $M_1$ is not empty and that $\bdel \tau$ is not identically zero in $M$. Based on Proposition \ref{glob-1} we show how this leads to a contradiction by distinguishing the following cases:
\begin{itemize}
\item[i)]$K_1=0.$ Observe first that $M_1$ must be dense in $M_0$ (thus in $M$). Indeed if $g$ is Einstein on some open region in $M_0$ it must be Ricci flat on that region by Proposition \ref{Ein-1}, (i). By (iv) in the same Proposition it follows that $\tau=0$ on some open subset in $M_0$ which is a contradiction with the definition of the latter.\\
Note that, locally in $M_2$, the moment map $z$ is not constant and that $C_2 >0$. Thus if $p=\frac{kC_2}{m}z$ is constant it must vanish identically, i.e. 
$p=0$ on $M$ and $k=0$.

To obtain a contradiction with having $M_1$ non-empty we proceed as follows. Computation on connected components of $M_2$ yields 
\begin{equation*}
\begin{split}
(dJd) \vert \varphi \vert^2=&C_2(z(dJd)z+dz \w Jdz)=-C_2(z \bb{X}^{\prime} \omega^{\V}+\bb{X}\omega^{\H}+\bb{X}\omega^{\V})\\
=&- \vert \tau \vert^2((m+2)\omega^{\V}+\omega^{\H}).
\end{split}
\end{equation*}
We have succesively used the expressions for the norms of $\varphi, \tau$ in Lemma \ref{lengths}, equation \eqref{der-K} as well as 
$\bb{X}=C_1z^{m+1}$, since $k=0$.
It follows that $((dJd) \vert \varphi \vert^2)^{\w m}=(-1)^m \vert \tau \vert^{2m}  (m+2)\, \omega^m$ on $M_1$ and hence, by the density arguments above, on $M$. The desired contradiction follows 
by applying Stokes' theorem.

\medskip

\item[ii)]$K_1$ does not vanish identically. Let $x_{+}$ be a maximum point for $p$, with $m_{+}=p(x_{+})$. Then 
$\di\!p$ vanishes at $x_{+}$ hence by equation \eqref{max} we obtain $m_{+}(\Delta p)(x_{+})=2k_1 m_{+}$. 
The maximum principle implies  $(\Delta p)(x_{+})\leq 0$.
Thus, if $m_{+} \neq 0$, we get that $k_1=0$. Integrating equation \eqref{max} over $M$ shows then  $\di\!p=0$, a contradiction. 
If $m_{+}=0$ then $p\leq 0$ over $M$; since $k_1 \geq 0$ integration in \eqref{max} followed by a positivity argument leads again to the vanishing of $
K_1$, a contradiction.
\end{itemize}

\medskip

\section{Hermitian Killing forms on Calabi type manifolds} \label{ex}
\subsection{Calabi-type metrics} \label{defn-cal}
We briefly recall the construction of Calabi-type metrics on total spaces of complex bundles. Let $\bb{C}=\bb{R}^2$ be equipped with real co-ordinates 
$x,y$ and complex structure $J_{\bb{C}}$ determined from $J_{\bb{C}}\partial_x=-\partial_y$. 
We work with metrics $h$ on $\bb{C}$(or some open part of it) with respect to which $J_{\bb{C}}$ is orthogonal and the $\bb{S}^1$-action given by complex multiplication is holomorphic and Hamiltonian,thus isometric. We require that the moment map for the circle action, determined from 
$K_{\bb{C}} \lrcorner \omega_h=\di\! z_h$ satisfies $z_{h}>0$. Note that $K_{\bb{C}}=-y\partial_x+x\partial_y$.

Let $(N^{2(m-1)},g_N, J_N),m \geq 3$ be a K\"ahler manifold with K\"ahler form $\omega_N=g_N(J_N\cdot, \cdot )$. Let $L \to N$ be a Hermitian line bundle equipped with a Hermitian connection $D$ such that $R^{D}=i\omega_N \otimes 1_L$ with the curvature convention $R^{D}(X,Y)=-D^{2}_{X,Y}+D^{2}_{Y,X},\ X,Y\in TN$. 

The manifold $M=L$ has a natural circle 
action induced by complex multiplication in the fibers; we denote by $K$ its 
infinitesimal generator. Then $M$ can be recovered as the associated bundle $P\times_{\bb{S}^{1}}\bb{C}$, where $P$ is the sphere bundle of $L$ and the free circle action on $P\times \bb{C}$ is $(m,w)z=(mz, wz^{-1})$. Indicating with $\V$ the real rank two distribution 
tangent to the fibers of $L$ we obtain a direct sum splitting $TM=\V \oplus \H$ as follows. The connection in the
principal bundle 
$\bb{S}^1 \hookrightarrow P \times \bb{C} \stackrel{p}{\to} M$
has horizontal distribution $\mathcal{H}^P\oplus T\bb{C}$, where 
 $\mathcal{H}^P=\ker \Theta^{P}$ and $\Theta^{P}$ is the connection in $P$ induced by $D$.  In this picture, $\V$ respectively $\H$ are the projection of $T\bb{C}\subset T(P\times \bb{C})$ respectively $\mathcal{H}^P\subseteq T(P \times \bb{C})$ down to $TM$.

The canonical complex structure $J$ on $M$, preserving $\V$ and $\H$, is obtained as follows: on $\V$ it is induced by the canonical complex structure on $\mathbb{C}$ while on $\H$ it is the horizontal lift of the complex structure from $N$. 

A Calabi-type metric on $M$ is a circle invariant Riemannian metric $g$, K\"ahler w.r.t $J$ essentially caracterised by the following requirements. The K\"ahler form  $\omega=\omega^{\V}+\omega^{\H}$ according to $TM=\V \oplus \H$ where 
\begin{equation*}
p^{\star}\omega^{\V}=\omega_h+\Theta^P \w \di\!z_h, \ \omega^{\H}=z\omega_N
\end{equation*}
and $z=p^{\star}z_h:M \to (0,\infty)$ is a moment map for the canonical circle action, i.e. $K \lrcorner \, \omega=\di\!z$. Moreover we impose that the length function $g(K,K)=\bb{X}(z)$ for some smooth function $\bb{X}:(0,\infty) \to [0,\infty)$.
W.r.t. $\V$ the K\"ahler manifold $(M,g,J)$ is of Calabi type in the sense of Definition \ref{C-T}. Indeed it is easy to check that $\V$ is totally geodesic, holomorphic and conformal, with Lee form $\theta=\di\! \ln z$.

\subsection{Lifts} \label{lifts}
We determine, on K\"ahler manifolds $(M,g,J)$ as constructed above, the structure of the space $\HK^{0,m-1}(M,g)$. In particular, we will make explicit 
instances when $(g,J)$ admits 
co-closed Hermitian Killing forms. The main ingredient in this computation is to show how to lift, canonically, sections of $\Lambda^{p,q}(N, L)$ to differential forms on 
$M$ and how to compare the Dolbeault operators $\del: \Lambda^{p,q}(N,L) 
\to \Lambda^{p+1,q}(N,L)$ and $\del:\Lambda^{p,q}M \to \Lambda^{p+1,q}M$.

For sections $\gamma\in \Lambda^{p,q}(N, L^k), k \in \bb{Z}$ we indicate with $\tg$ the canonical lift of $\gamma$ to $\Lambda^{p+q}(P,\bb{C})$. 
\begin{lema} \label{lift-m}
We have a well defined map $\Lambda^{p,q}(N,L^k) \to \Lambda^{p,q}\H, \gamma \mapsto \hat{\gamma}$ uniquely determined from 
\begin{equation} \label{master} 
 p^{\star}\,\hg=\ol{w}^k\,\tg \ \mbox{if} \ k \geq 0 \ \mbox{respectively} \ p^{\star}\hat{\gamma}=w^{-k}\,\tg \ \mbox{if} \ k \leq 0.
\end{equation} 
\end{lema}
\begin{proof}
From the definitions the lift $\tilde{\gamma}$ of $\gamma \in \Lambda^{p,q}(N,L^k)$ to $P$ is 
$-k$-equivariant, that is $R^{\star}_{z}\,\tg=z^{-k}\,\tg$, where $(R_{z})_{z \in \bb{S}^1}$ denotes the principal circle action in $P$. It follows that the form $\ol{w}^k\,\tg \in\Lambda^{p+q}(P\times \bb{C})$ if $k\geq 0$(or $w^{-k}\,\tg \in\Lambda^{p+q}(P\times \bb{C})$ if $k\leq 0$) is circle-invariant with respect to the principal circle action on $P \times \bb{C}$; since it moreover vanishes on $T\bb{C}$, it projects onto $M$ according to \eqref{master}. By construction $\hg$ belongs to $\Lambda^{p,q}\H$. 
\end{proof}
We denote by $\di\!:\Lambda^q(N,L^k) \to \Lambda^{q+1}(N,L^k)$ the exterior derivative coupled with the connection $D$ in $L$. The orthogonal projection onto $\Lambda^{\star}\H$ of the ordinary exterior derivative, acting on $\Lambda^{\star}\H$, will be denoted by $\di_{\H}$.
\begin{pro} \label{lifty} 
The lifting map $\gamma \mapsto \widehat{\gamma}$ has the following properties 
\begin{itemize}
\item[(i)] for any $\gamma $ in $\Lambda^{p,q}(N,L^k)$ the lift $\hg$ belongs to $\Lambda^{p,q}\H$ and 
$$\li_{K}\hg=-ik\hg, \ \li_{JK}\hg=k\hg, \ \di_{\H}\!\hg=\widehat{\di\! \gamma}$$
\item[(ii)] any $\beta $ in $\Lambda^q \H$ such that $\li_K\beta=-ik\beta, \ 
\li_{JK}\beta=k\beta$ satisfies $\beta=\hg$ for some $\gamma$ in 
$\Lambda^q(N,L^k)$
\item[(iii)] we have 
\begin{equation} \label{del-lift}
\begin{split}
&\bdel \widehat{\gamma}=\widehat{\bdel \gamma}\\
&\bb{X} \del \widehat{\gamma}=2k\del z \w \widehat{\gamma}+\bb{X} \widehat{\del \gamma}
\end{split}
\end{equation}
whenever $\gamma$ is in $\Lambda^{p,q}(M,L^k)$.
\end{itemize}
\end{pro}
\begin{proof}
(i) to see that $\hg\in\Lambda^{p,q}\H$, notice that the complex structure $J$ on $M$ can, alternatively,  be recovered by projecting down to $M$ the sum of  $J_{\bb{C}}$ and the lift of $J_{N}$ to $\mathcal{H}$. Differentiation in \eqref{master} while assuming $k \geq 1$ leads to
\begin{equation} \label{lf1}p^{\star}(\di\! \hg)=\ol{w}^k\di\! \tg+k\ol{w}^{k-1}\di\! \ol{w} \w \tg=p^{\star}(\widehat{\di\!\gamma})+k\ol{w}^{k-1}
(\di\! \ol{w}-i\ol{w} \Theta^{P})\w \tg,
\end{equation}
by using the standard formula $\di\! \tg=\widetilde{\di\! \gamma}-i\Theta^{P} \w \tg$. Since 
$\di\! \ol{w}-i\ol{w} \Theta^{P}$ in $\Lambda^1(P \times \bb{C})$ is horizontal and vanishes on $\mathcal{H}$ the last part of the claim follows by projection onto $M$. The expressions for the Lie derivatives follow by evaluation of the displayed formula on $-y\del_x+x\del_y, x\del_x+y\del_y$ which we recall are the horizontal lifts of $K,JK$ to $P \times \bb{C}$.\\
(ii) the pull back $\alpha=p^{\star}\beta$ belongs, by construction, to $\Lambda^q\mathcal{H}^P \subseteq \Lambda^q(P \times \bb{C})$. Lifting $K,JK$ horizontally to $P \times \bb{C}$
the set of requirements on $\beta$ becomes 
$\li_{-y\del_x+x\del_y} \alpha=-ik\alpha, \ \li_{x\del_x+y\del_y} \alpha=k\alpha.$  
Then $\overline{w}\li_{\del_x}\alpha=k\alpha, \  \overline{w}\li_{\del_y}\alpha=-ik\alpha$ showing that $\frac{1}{\overline{w}^k}\alpha$ is constant in direction of $\bb{C}$ for 
$w \neq 0$. By continuity $\alpha=\ol{w}^k\,\sigma$ where $\sigma$ is in $\Lambda^qP$. Since $\alpha$ is invariant w.r.t to the circle action on $P \times \bb{C}$, we have $R^{*}_{z}\,\sigma=z^{-k}\,\sigma, z\in \bb{S}^{1}$. Tautologically $\sigma=\tg $ with $\gamma$ in $\Lambda^q(N,L^k)$ and the claim is proved.\\
(iii) follows from the definition of the lifting map and \eqref{lf1}.
\end{proof}
It is well known that letting $I:=-J_{\vert \V}+J_{\H}$ yields an integrable almost complex structure, orthogonal w.r.t. $g$. The Chern connection 
$\nabla^c=\nabla+\frac{1}{2}(\nabla I)I$ of $(g,I)$ turns out to be the projection of the Levi-Civita connection of $g$ onto the splitting $TM=\V
\oplus \H$(see \cite{NaOr} for details and more general results). We have 
\begin{equation} \label{eqn-c}
\nabla^c_{X^{\H}}Y^{\H}=(\nabla^{g_N}_XY)^{\H}
\end{equation} for all $X,Y$ in $TN$, where $\nabla^{g_N}$ is the Levi-Civita connection of $g_N$ and where $X^{\H}$ is the horizontal lift of $X$ to $\H$. Then 
\begin{equation} \label{rel-connex} 
\nabla^c_{X^{\H}}\widehat{\alpha}=\widehat{\nabla_X \alpha}
\end{equation}
where $X$ is in $TN$, $\alpha$ is a section of $\Omega^{p,q}(N,L^{k})$ and $\nabla$ is the coupled connection therein. Key to solving differential equations on $M$ is comparing 
differential operators on $M$ in terms of their horizontal and vertical counterparts.
\begin{lema} \label{nsplit-f}
The following hold
\begin{itemize}
\item[(i)] we have 
\begin{equation} \label{comp-H1}
\nabla_X\tau=\nabla^c_X\tau-\frac{i}{z}K^{01} \w (X \lrcorner \, \tau)
\end{equation} 
whenever $(X, \tau) \in \H \times \Lambda^{0,p}\H$ 
\item[(ii)] as well as 
\begin{equation} \label{liena}
\li_{K_{01}}=\nabla_{K_{01}}+\frac{ip\bb{X}}{2z}
\end{equation}
on $\Lambda^{0,p}\H$.
\end{itemize}
\end{lema}
\begin{proof}
By \eqref{der-K} we have $K \lrcorner \, \nabla_X\tau=-\nabla_XK \lrcorner \, \tau=-\frac{i}{2z}\bb{X}(z)X \lrcorner \, \tau$ and the claim follows easily. The claim in \eqref{liena} follows by using again \eqref{der-K} in the comparaison formula \eqref{lncom}.
\end{proof}
Indicating with $\bdel_{\H}$ the component on $\Lambda^{0,\star}\H$ of $\bdel$ acting on $\Lambda^{0,\star}\H$ a straightforward argument leads to 
\begin{equation} \label{bdelhh}
\bdel=\bdel_{\H}+\frac{2}{\bb{X}} K^{01} \w \li_{K_{01}} \ \mbox{on} \ \Lambda^{0,\star}\H.
\end{equation}
If $E$ is a holomorphic line bundle over $N$ we denote with  
\begin{equation*}
\HK^{0,p}(N,E):=\{\tau \in \Lambda^{0,p}(N,E):\nabla^{01}\tau=\frac{1}{p+1}\bdel \tau\}
\end{equation*}
the space of $E$-valued Hermitian Killing $(0,p)$-forms, $p \geq 1$. When $p=0$ we define 
$\HK^{0,0}(N,E):=\{s \in \Gamma(E): \bdel s \in \HK^{0,1}(N,E)\}$. We denote with $L^k,k \in \bb{Z}$ the $k$-th tensor power of $L$.
\begin{teo} \label{ex-calnew}
We have an isomorphism 
\begin{equation*}
\bigoplus_{k \in \bb{Z}}\biggl [ \HK^{0,p-1}(N,L^k) \oplus \HK^{0,p}(N,L^k) \biggr ] \to \HK^{0,p}(M,g), \ (\tau_1, \tau_2) \mapsto \bdel(z^p\widehat{\tau_1})+z^{p+1} \widehat{\tau_2}
\end{equation*}
whenever $p \geq 1$. 
\end{teo}
\begin{proof}
Pick $\tau \in \HK^{0,p}(M,g)$; since the operator $\nabla^{01}-\frac{1}{p+1}\bdel$ is invariant under $\li_K$ after expanding $\tau$ as a Fourier series we may further assume that 
$\li_K \tau=ik \tau $ with $k \in \bb{Z}$. Write $\tau=K^{0,1} \w \tau_1+\tau_2$ off the zero set of $K$. Then 
\begin{equation*}
\bdel \tau=K^{01} \w \bdel \tau_1+\bdel \tau_2=-K^{01} \w \bdel_{\H}\tau_1+\bdel_{\H}\tau_2+\frac{2}{\bb{X}}K^{01} \w \li_{K_{01}}\!\tau_2
\end{equation*}
after expanding $\bdel$ according to \eqref{bdelhh}. Since $K$ is holomorphic we have $\nabla^{01}K^{01}=0$ 
thus the Hermitian Killing equation yields 
\begin{equation*}
K^{01} \w \nabla_{K_{01}}\tau_1+\nabla_{K_{01}}\tau_2=\frac{1}{p+1}K_{01} \lrcorner \bdel \tau=\frac{1}{p+1}(-\frac{\bb{X}}{2}\bdel_{\H}\tau_1+\li_{K_{01}}\!\tau_2).
\end{equation*}
The components of the displayed equation on $K^{01} \w \Lambda^{0,p-1}\H$ respectively $\Lambda^{0,p}\H$ thus read 
\begin{equation*}
\nabla_{K_{01}}\tau_1=0, \ \nabla_{K_{01}}\tau_2=\frac{1}{p+1}(-\frac{\bb{X}}{2}\bdel_{\H}\tau_1+\li_{K_{01}}\!\tau_2).
\end{equation*}
In terms of the Lie derivative, using \eqref{liena} leads to 
\begin{equation*}
\li_{K_{01}}\! \tau_1=\frac{i(p-1)\bb{X}}{2z} \tau_1, \ \li_{K_{01}}\! \tau_2-\frac{i(p+1)\bb{X}}{2z} \tau_2=-\frac{\bb{X}}{2p}\bdel_{\H}\tau_1.
\end{equation*}
It follows easily that $\li_{K_{01}}\!(z^{-(p-1)}\tau_1)=0$ thus $\tau_1=z^{p-1}\widehat{\alpha_1}$ with $\alpha_1 \in \Lambda^{0,p-1}(N,L^k)$, in particular $\tau_1$ is globally defined and $\bdel_{\H}\tau_1=z^{p-1}\widehat{\bdel \alpha_1}$. Having $\tau_1$ thus parametrised leads easily to 
$\tau_2=-\frac{i}{p}z^p\widehat{\bdel \alpha_1}+z^{p+1}\widehat{\alpha_2}$ with $\alpha_2 \in \Lambda^{0,p-1}(N,L^k)$. In other words 
\begin{equation*}
\tau=-\frac{i}{p}\bdel(z^p \widehat{\alpha_1})+z^{p+1}\widehat{\alpha_2}.
\end{equation*}
There remains to examine the Hermitian Killing equation on $\tau$ in direction of $\H$. Since $\bdel \tau=z^{p+1}\widehat{\bdel \alpha_2}+\bdel(z^{p+1}) 
\w \widehat{\alpha_2}$ the latter reads 
\begin{equation*}
K^{01} \w \nabla^{01}_X \tau_1+\nabla^{01}_X \tau_2=\frac{1}{p+1}z^{p+1} X \lrcorner \widehat{\bdel \alpha_2}-\frac{1}{p+1} \bdel(z^{p+1})\w (X \lrcorner \widehat{\alpha_2}) , \ X \in \H
\end{equation*}
after also taking into account that $\nabla^{01}K^{01}=0$. Identifying components whilst simultaneously using Lemma \ref{nsplit-f} yields 
\begin{equation*}
\nabla_{X_{01}}^c\tau_2=\frac{z^{p+1}}{p+1}X \lrcorner \widehat{\bdel \alpha_2}, \ \nabla_{X_{01}}^c\tau_1=\frac{i}{z}X \lrcorner (\tau_2-z^{p+1}\widehat{\alpha_2}).
\end{equation*}
Projected down on $N$ via Lemma \ref{nsplit-f} the second equation reads $\nabla^{01}\alpha_1=\frac{1}{p}\bdel \alpha_1$. Similarly the first has two 
components with coefficients $z^p$ and $z^{p+1}$ thus we arrive at $\nabla^{01}(\bdel \alpha_1)=0$ and $\alpha_2 \in \HK^{0,p}(N,L^k)$. When $p=1$ this simply states that $\alpha_1 \in \HK^{0,0}(N,L^k)$. When $p \geq 2$ we know by Proposition \ref{prol-1},(i) that $\bdel \alpha_1 \in \HK^{0,p}(N,L^k)$ and the proof is complete.
\end{proof}
At this stage a few remarks are in order. If $E$ is a holomorphic line bundle over $N$ we denote with $H^{0,p}(N,E)$ the space 
of $E$-valued holomorphic forms of type $(0,p)$ and let $K_N:=\Lambda^{0,m-1}N$. 
\begin{rema} \label{complete}
\begin{itemize}
\item[(i)]
Theorem \ref{ex-calnew} provides many examples of K\"ahler structures carrying Hermitian Killing forms. Via the identification $\HK^{0,m-1}(N,L^k)=H^0(N,K_N \otimes L^k)$ it provides an injection $\bigoplus_{k \in \bb{Z}}H^0(N,K_N \otimes L^k) \to \HK^{0,m-1}(M,g)$ and also proves Theorem \ref{main3} in the introduction.
\item[(ii)]whilst the examples above are not compact, no constraint on the momentum profile is imposed; if the latter is choosed to grow at most quadratically \cite{Hw-S} the metrics under consideration are complete.
\end{itemize}
\end{rema}
As far as co-closed Hermitian Killing forms are concerned we can make the following 
\begin{pro} \label{ex-cal} 
Let $\tau\in \HK^{0,m-1}(M,g) \cap \ker \di^{\star}$ be such that $\tau_{\vert \V}=0$. Then, assuming that $m \geq 3$
\begin{itemize}
\item[(i)] $\rho^N=k\omega_N$ where $k \in \mathbb{Z}$(and hence $K_N  \cong L^{-k}$)
\item[(ii)] up to multiplication by a constant $\tau=z^m \widehat{id}$ where  
$id \in \Lambda^{0,m-1}(N,L^{-k})$ is induced by the vector bundle isomorphism above.
\end{itemize}
\end{pro}
\begin{proof}
From the K\"ahler identities having $\di^{\star}\!\tau=0$ is equivalent with $L^{\star}_{\omega}(\del \tau)=0$ that is $L^{\star}_{\omega}(\del_{\H}\tau)=0$. Since 
$L_{\omega}:\Lambda^{0,m-2}\H \to \Lambda^{1,m-1}\H$ is an isomorphism it follows that $\del_{\H}\tau=0$. This allows writing $\tau=z^{m}\tau_0$ where 
$\tau_0 \in \Lambda^{0,m-1}\H$ satisfies $\di_{\H}\tau_0=0$ and $\li_{K_{01}}\! \tau_0=0$. Decomposing $\tau_0$ as a Fourier series, or locally as 
as holomorphic series in $w \in \bb{C}$, the coefficients $\tau_0^k=\widehat{\gamma_k}$ where $\gamma_k \in \Lambda^{0,m-1}(N,L^{-k}),k \in \bb{Z}$. Thus $\nabla \gamma_k=0$ according to Theorem \ref{ex-calnew}. If $\gamma_k$ is not zero it provides an isomorphism between $K_N$ and $L^{-k}$ hence $k$ is uniquely determined from $\rho^N=k\omega_N$ and both claims are proved.
\end{proof}

$\\$
{\bf{End of Proof of Theorem \ref{main1}}}\\
Proposition \ref{loc-descr} fully describes the local geometry of $(g,J)$ as well as the expression for $\varphi$ in terms of $\tau$. To complete the argument there remains to describe $\tau$, which by Proposition \ref{pro1} is in the space $\HK^{0,m-1}(M,g) \cap \ker \di^*$.
Since the arguments in the proof of Proposition \ref{ex-cal},(ii), are purely local we still have  $\tau=z^m \widehat{id}$.\\

We conclude this section by establishing the maximal domain of definition of the solutions $(\varphi,\tau)$ to equation \eqref{special-i} constructed in Proposition \ref{loc-descr}. Clearly the form $\tau=z^m \widehat{id}$ is defined on the whole of $M=L$. 
\begin{pro}\label{ex-extn} 
Let $M=K_N^{-\frac{1}{k}}$ with $k \in \bb{Z}^{\times}$. Consider the metric $g$ with momentum profile $\bb{X}=z(C_1z^m+\frac{2k}{m}),C_1 \in \bb{R}$ together with $\varphi=\frac{z^{m+1}}{\bb{X}(z)}\del z \w \widehat{id}$. The maximal domain of definition for the pair $(g,\varphi)$ is 
\begin{itemize}
\item[(i)] $M^{\times}$ if $k >0>C_1$
\item[(ii)]  $\{r >a\}$ with $a=(C_1\lambda)^{-\frac{1}{2k}}$ if $C_1>0>k$ 
\item[(iii)]  $\{0 <r <a\}$ with $a=(C_1\lambda)^{-\frac{1}{2k}}$ if $C_1,k >0$ 
\item[(iv)] $M^{\times}$ if $C_1=0$.
\end{itemize}
\end{pro}
\begin{proof} 
To deal with $g$ and $\varphi$ we recall a few facts concerning the metrics $\omega_h$ on $\bb{C}$. We have $\omega_h=-(\di\! J_{\bb{C}}\di \!)F(r)$ where $r^2=x^2+y^2$. The moment map is given by $z_h=G(r)$ with $G(r)=rf^{\prime}(r)$ thus the positivity conditions are $G,G^{\prime}>0$. Thus 
$h(K_{\bb{C}},K_{\bb{C}})=rG^{\prime}(r)$; the desired momentum profile corresponds to $G$ solving the ODE $rG^{\prime}(r)=G(r)(C_1G^m+\frac{2k}{m})$. The solutions are 
\begin{equation*}
G^m(r)=\frac{2k\lambda r^{2k}}{m(1-C_1\lambda r^{2k})}
\end{equation*}
where $\lambda>0$. Also record that 
\begin{equation*}
\frac{z^{m+1}}{\bb{X}(z)}\del z=G^{m+1}(r)\del \ln r.
\end{equation*}
All statements follow now easily from these facts. Note that the maximal domain for $\varphi$ alone can be larger than that for $g$ and also that 
$g$ is incomplete for $m \geq 3$ since then $\bb{X}$ does not grow quadratically \cite{Hw-S}.
\end{proof}


\begin{thebibliography}{99}

\bibitem{AD}
A. Andrada, I. G. Dotti,
\textit{Conformal Killing-Yano 2-forms},
Differential Geom. Appl.  {\bf 58} (2018), 103--119.

\bibitem{ADM}
V. Apostolov, T. Dr\u{a}ghici, A. Moroianu, 
\textit{A splitting theorem for K\"ahler manifolds whose Ricci tensors have constant eigenvalues},
Internat. J. Math. {\bf{12}} (2001), no. 7, 769--789.
\bibitem{ACG-C}
V. Apostolov, D. Calderbank, P. Gauduchon,
\textit{The geometry of weakly self-dual K\"ahler surfaces},  Compositio Math. {\bf{135}} (2003), no.3, 279--322.
\bibitem{ACG}
V. Apostolov, D. Calderbank, P. Gauduchon, 
\textit{Hamiltonian $2$-forms in K\"ahler geometry. I.General theory},  J. Differential Geom.  {\bf{73}} (2006),  no.3, 359-412.
\bibitem{Baer}
Ch. B\"ar, \textit{Real Killing spinors and holonomy}, Comm. Math. Phys. {\bf{154}} (1993), no.3, 509--521.

\bibitem{BMS}
F. Belgun, A. Moroianu, U. Semmelmann,
\textit{Killing forms on symmetric spaces}, 
Differential Geom. Appl.  {\bf 24} (2006), no. 3, 215--222.


\bibitem{Benn}
I.M. Benn, P. Charlton, 
\textit{Dirac symmetry operators from conformal Killing-Yano tensors},
Classical Quantum Gravity {\bf 14} (1997), 1037--1042. 

\bibitem{Branson}
Th. Branson, 
\textit{Stein-Weiss operators and ellipticity}, 
J. Funct. Anal. {\bf 151} (1997), no. 2, 334--383. 

\bibitem{C}
E. Calabi, \textit{Extremal K\"ahler metrics}, Seminar on differential geometry, Annals of Mathematics Studies 102 (Princeton University Press, Princeton, NJ, 1982) 259--290. 
\bibitem{Chiossi-Nagy}
S.G. Chiossi, P.-A. Nagy, 
\textit{Complex homothetic foliations on K\"ahler manifolds},  Bull. Lond. Math. Soc. {\bf{44}} (2012), no. 1, 113--124. 

\bibitem{Der-M}
A. Derdzinski, G. Maschler, 
\textit{Local classification of conformally-Einstein K\"ahler metrics in higher dimensions},
Proc. London Math. Soc. (3) {\bf{87}} (2003), no. 3, 779--819.

\bibitem{GS}
N. Ginoux, U. Semmelmann,
\textit{Imaginary K\"ahlerian Killing spinors},
Ann. Global Anal. Geom. {\bf 40} (2011), no. 4, 467--495.
 

\bibitem{Hw-S}
A. Hwang, M. Singer, \textit{
A momentum construction for circle-invariant K\"ahler metrics},
Trans. Amer. Math. Soc. {\bf{354}} (2002), no.6, 2285--2325. 

\bibitem{Kirchberg}
K.-D. Kirchberg, 
\textit{Killing spinors on K\"ahler manifolds},
Ann. Global Anal. Geom. {\bf 11} (1993), no. 2, 141--164.


\bibitem{Ku}
W. K\"uhnel, \textit{Conformal transformations between Einstein spaces}, Conformal geometry (Bonn, 1985/1986), 105--146, Aspects Math.,E12,Vieweg, Braunschweig,1988.

\bibitem{MaRo}
V. Matveev, S. Rosemann, \textit{Conification construction for K\"ahler manifolds and its application in $c$-projective geometry},
Adv. Math. {\bf{274}} (2015), 1--38.

\bibitem{A19}
A. Moroianu
\textit{Conformally related Riemannian metrics with non-generic holonomy}, 
J. Reine Angew. Math., {\bf 755} (2019), 279--292.

\bibitem{AUT}
A. Moroianu, U. Semmelmann, \textit{Twistor forms on K\"ahler manifolds},  Ann.Sc.Norm.Super.Pisa Cl. Sci.(5) {\bf{2}} (2003),  no. 4, 823-845.

\bibitem{AU2}
A. Moroianu, U. Semmelmann, 
\textit{Killing forms on quaternion-K\"ahler manifolds}, 
Ann. Global Anal. Geom. {\bf 28} (2005), no. 4, 319--335.

\bibitem{AP}
A. Moroianu, P. Gauduchon, \textit{Killing 2-forms in dimension 4. Special metrics and group actions in geometry}, 161--205,
Springer INdAM Ser., 23, Springer, Cham, 2017.

\bibitem{handbk}
P.-A. Nagy, \textit{Connections with totally skew-symmetric torsion and nearly-K\"ahler geometry},  \textit{Handbook of pseudo-Riemannian geometry and supersymmetry},  347--398,
IRMA Lect. Math. Theor. Phys., {\bf{16}}, Eur. Math. Soc., Z\"urich, 2010.

\bibitem{NaOr}
P.-A. Nagy, L. Ornea, 
\textit{Conformal foliations, K\"ahler twists and the Weinstein construction}, \url{http://de.arxiv.org/pdf/1909.11499.pdf}.

\bibitem{Penrose1}
M. Walker, R. Penrose,
\textit{On quadratic first integrals of the geodesic equations for type \{22\} spacetimes}, 
Comm. Math. Phys. {\bf 18} (1970), 265--274.


\bibitem{Mas}
M. Pontecorvo, \textit{On twistor spaces of anti-self-dual Hermitian surfaces},
Trans. Amer. Math. Soc. {\bf{331}} (1992), no. 2, 653--661.

\bibitem{uwe} 
U. Semmelmann, 
\textit{Conformal Killing forms on Riemannian manifolds}, Math. Z. {\bf{245}} (2003), no. 3, 503--527. 

\bibitem{uwe2}
U. Semmelmann, 
\textit{Killing forms on $\mathrm{G}_2$- and $\mathrm{Spin}_7$-manifolds},
J. Geom. Phys. {\bf 56} (2006), no. 9, 1752-1766.

\bibitem{Tash}
Y. Tashiro, 
\textit{Complete Riemannian manifolds and some vector fields}, Trans. Amer. Math. Soc. {\bf 117} (1965) 251--275. 

\bibitem{Tonn}
C. T\o nnesen-Friedman, 
\textit{Extremal K\"ahler metrics and Hamiltonian functions. II.}, 
Glasg. Math. J. {\bf{44}} (2002), no. 2, 241--253.

\bibitem{V}
I. Vaisman, 
\textit{Some curvature properties of complex surfaces}, 
Ann. Mat. Pura Appl. (4) {\bf{132}} (1982), 1--18. 


\bibitem{Penrose}
M. Walker, R. Penrose,
\textit{On quadratic first integrals of the geodesic equations for type \{22\} spacetimes}, 
Comm. Math. Phys. {\bf 18} (1970), 265--274.

\bibitem{Yano}
K. Yano,
\textit{Some remarks on tensor fields and curvature},
Ann. Math. (2) {\bf 55} (1952), 328--347.

\end{thebibliography}
\end{document}